\documentclass[psamsfonts]{amsart}

\usepackage[english]{babel}

\usepackage{amssymb}
\usepackage{mathrsfs}
\usepackage{hyperref}
\usepackage{xcolor}
\usepackage{eucal}

\usepackage{amssymb}
\usepackage{mathabx}
\usepackage{amsmath}
\usepackage{amsthm}
\usepackage{imakeidx}
\usepackage{breqn}

\usepackage{diagbox}

\usepackage{verbatim}

\usepackage{fancyhdr}

\usepackage{tikz-cd}

\usepackage[utf8]{inputenc}

\usepackage{graphicx}

\usepackage{xcolor}

\usepackage{stmaryrd}

\usepackage[style=alphabetic, url=false, doi=false, backend=bibtex]{biblatex}

\theoremstyle{definition}

\newtheorem{mydef}{Definition}[section]

\newtheorem{myque}[mydef]{Question}

\theoremstyle{remark}

\newtheorem{mybem}[mydef]{Remark}

\theoremstyle{plain}

\newtheorem{mycol}[mydef]{Corollary}
\newtheorem{mysen}[mydef]{Theorem}
\newtheorem{mylem}[mydef]{Lemma}
\newtheorem{mypro}[mydef]{Proposition}

\newtheorem*{myclaim}{Claim}

\numberwithin{mydef}{section}

\DeclareMathOperator{\topo}{top}

\DeclareMathOperator{\cof}{cof}
\DeclareMathOperator{\dom}{dom}

\DeclareMathOperator{\im}{im}

\DeclareMathOperator{\otp}{otp}
\DeclareMathOperator{\cf}{cf}

\DeclareMathOperator{\Add}{Add}

\DeclareMathOperator{\Coll}{Coll}

\DeclareMathOperator{\acc}{acc}

\DeclareMathOperator{\CH}{\mathsf{CH}}

\DeclareMathOperator{\RCS}{RCS}

\DeclareMathOperator{\Hull}{Hull}

\newcommand{\dB}{\mathbb{B}}

\newcommand{\dJ}{\mathbb{J}}

\newcommand{\dP}{\mathbb{P}}
\newcommand{\dQ}{\mathbb{Q}}

\newcommand{\uhr}{\upharpoonright}

\newcommand{\ZFC}{\mathsf{ZFC}}

\title{A Topological Rainbow Ramsey Theorem} 

\author{Hannes Jakob}
\address{University of North Texas,
Denton, Texas, USA, 76201}
\email{hannes.jakob@unt.edu}

\author{Jing Zhang}
\address{University of North Texas,
Denton, Texas, USA, 76201}
\email{jing.zhang2@unt.edu}

\subjclass[2020]{Ramsey theory, rainbow, semiproper forcing, large cardinal, strong Chang's conjecture} 

\date{\today}

\begin{document}
	
	\maketitle
	
	\begin{abstract}
	We show that it is consistent relative to the existence of suitable large cardinals that for any countable-to-one coloring $c: [\omega_2]^2\to \omega_2$, there exists a closed subset $A\subseteq \omega_2$ of order type $\omega_1$ such that $c\restriction [A]^2$ is injective. This theorem simultaneously strengthens two theorems, one by Abraham, Cummings and Smyth and another one by Garti and Zhang, as well as answers a question raised by Garti and Zhang. New combinatorial principles involving towers of countable elementary submodels, games concerning regressive functions and variants of strong Chang's conjecture, which are key elements of the proof, are investigated.
	\end{abstract}
	
 \section{Introduction}
 
 A typical problem in rainbow Ramsey theory is the following: given a coloring $c:[\kappa]^n\rightarrow\theta$ satisfying that each color is not used ``too many times", we are asked to find a ``large" $y\subseteq\kappa$, such that $y$ is $c$-\emph{rainbow}, namely, $c\upharpoonright[y]^n$ is one-to-one. 
 The exact meaning of ``too many times" and ``large" varies depending on the context. Let us introduce a notation that is central to our discussion.
 
 \begin{mydef}\label{rainbowinitialdef}
Let $\lambda, \kappa, n$ be ordinals and $i$ be a cardinal.
We use $\lambda\to^{*} (\kappa)^n_{<i-bdd}$ to abbreviate: For any $f: [\lambda]^n \to \lambda$ that is $<i$\emph{-bounded}, namely for any $\alpha\in \lambda$, $|f^{-1}\{\alpha\}| < i$, there exists an $f$-rainbow $A\subset \lambda$ of order type $\kappa$. Let $\lambda\to^{*} (\kappa)^n_{i-bdd}$ denote $\lambda\to^{*} (\kappa)^n_{<i^+-bdd}$
\end{mydef}

Note that $\omega\to^* (\omega)^2_{2-bdd}$ is a consequence of the infinite Ramsey theorem \cite{Ramsey}, namely for any $c: [\omega]^2\to 2$, there exists an infinite $H\subseteq \omega$ and $i<2$, such that $f\restriction [H]^2= \{i\}$. The Ramsey theorem is usually denoted as $\omega\to (\omega)^2_2$. As a result, rainbow Ramsey theory is sometimes called \emph{sub-Ramsey} theory, especially in finite combinatorics. 

The rainbow variations have been widely considered in the literature. For example, they are studied in finite combinatorics \cite{MR867747, rainbowsurvey}, in computability theory, \cite{MR2583822, Wangwei}, and in combinatorics on countably infinite structures and ultrafilters on a countable set, \cite{MR3518438,MR3135506}. Our focus in this paper is the rainbow Ramsey theory in the context of infinite, often times uncountable, sets. Such a study was initiated by Galvin \cite{GalvinLetter} (cited in \cite{TodorcevicPositivePartition}), where he showed that $\CH$ implies that $\omega_1\not\to^* (\omega_1)^2_{2-bdd}$ and asked if $\omega_1\to^* (\omega_1)^2_{2-bdd}$ is consistent. Note that here the corresponding Ramsey statement is of no help since Sierpinski \cite{Sierpinski} showed that $\omega_1\not\to (\omega_1)^2_{2}$. In 1983, Todorcevic \cite{TodorcevicPositivePartition} showed that $\omega_1\to^* (\omega_1)^2_{2-bdd}$ is indeed consistent relative to the consistency of $\ZFC$. In fact, in his model, the following strengthening is true: for any $c: [\omega_1]^2\to \omega_1$ such that each $\alpha\in \omega_1$, $c^{-1}[\{\alpha\}]$ is finite, there exists a partition $f: \omega_1\to \omega$ such that for each $i\in \omega$, $f^{-1}[\{i\}]$ is a $c$-rainbow subset of $\omega_1$. The strengthening is two-fold: 1) the boundedness condition of the coloring is more relaxed and 2) the conclusion is stronger, namely the rainbow set we can find is larger (for example, it can be stationary in $\omega_1$).

Let us explain another key word, ``topological", in the title. Ramsey-type problems involving topological spaces take the following form: given a topological space $X$, and a coloring $c: [X]^n \to \theta$, we want to find a ``large" subspace $Y$ satisfying certain topological constraints such that $c\restriction [Y]^n$ is ``small". These topological Ramsey problems have been considered by Baumgartner, Galvin, Van Douwen (see \cite{BaumgartnerTopological}) inspired by earlier theorems of Galvin and Devlin \cite{Devlin} regarding the big Ramsey degree of the rational numbers as a linear order. Recent advancements include a proof of a conjecture of Galvin, by Raghavan and Todorcevic \cite{GalvinConjecture} using large cardinals, and later by Inamdar \cite{inamdar2024ramseytheoremreals} in $\ZFC$, that states that for any finite coloring $c$ on pairs of $\mathbb{R}$, there is a subset $X$ homeomorphic to $\mathbb{Q}$ such that $c\restriction [X]^2$ has size at most $2$.

Let us discuss another family of topological Ramsey theorems that are closer to our discussion in this paper. 
We first introduce some notations for a more compact presentation. For linear orders $\phi, \xi$, any cardinal $\lambda$ and natural number $n\in \omega$, we let $\phi \to (\xi)^n_{\lambda}$ abbreviate that for any $c: [\phi]^n\to \lambda$, there is a subset $H\subseteq \phi$ isomorphic to $\xi$ as linear orders such that $c\uhr [H]^n$ is constant. The Baumgartner-Hajnal \cite{BaumgartnerHajnal} theorem states that 
for any linear order $\varphi$ satisfying $\varphi\to (\omega)^1_\omega$, for any $\alpha<\omega_1$ and $k\in \omega$, $\varphi\to (\alpha)^2_k$. Laver and Weiss \cite{Weiss} asked if the topological version of the Baumgartner-Hajnal theorem is true when $\varphi$ is taken to be $\omega_1$ with the order topology and $\mathbb{R}$ with the usual topology respectively. Schipperus \cite{Schipperus} answers both questions positively by showing that let $X$ be either $\omega_1$ or $\mathbb{R}$, then for any $c: [X]^2 \to k\in \omega$ and any $\alpha<\omega_1$, there exists a \emph{closed} subset $H\subseteq X$ of order type $\alpha+1$ such that $c\uhr [H]^2$ is constant. In other words, such $H$ is homeomorphic to $\alpha+1$ with the order topology. The next definition, which is central to this paper, combines the two variations we discussed above.

\begin{mydef}\label{defintion: topologicalrainbowRamsey}
Let $X, Y$ be topological spaces, $\kappa$ be a cardinal and $n\in \omega$. We let $X\to^* (\topo Y)^n_{<\kappa-bdd}$ abbreviate the following statement: for any $<\kappa$-bounded coloring $c: [X]^n \to X$, there exists $H\subseteq X$ homeomorphic to $Y$ such that $c\uhr [H]^n$ is injective.
\end{mydef}

When $X$, $Y$ are ordinals with the order topology, we can describe the partition relations above, as well as some variants, in a more combinatorial way: 

\begin{mydef}[Definition 0.2, \cite{GartiZhang}]
Let $\lambda$, $\alpha$ be ordinals and $\kappa$ be a cardinal. 
\begin{itemize}
\item $\lambda\to^{*} (\alpha-cl)^n_{<\kappa-bdd}$ abbreviates: for any $<\kappa$-bounded coloring $c: [\lambda]^n\to \lambda$, there exists a \emph{closed} $c$-rainbow subset of order type $\alpha$. Namely, $\lambda\to^{*} (\topo \alpha)^n_{<\kappa-bdd}$
\item Similarly, $\lambda\to^{*} (\alpha-st)^n_{<\kappa-bdd}$ abbreviates: for any $<\kappa$-bounded coloring $c: [\lambda]^n\to \lambda$, there exists a $c$-rainbow subset $A$ of order type $\alpha$ that is \emph{stationary in $\sup A$}. 

\end{itemize}
\end{mydef}

The first topological rainbow Ramsey theorem was proved by Todorcevic \cite{TodorcevicPositivePartition} where he showed, via an absoluteness argument, that in ZFC, $\omega_1\to^* (\alpha-cl)_{<\omega-bdd}$ for any $\alpha<\omega_1$. Notice that this theorem, proved before 1983, precedes the topological Baumgartner-Hajnal theorem proved by 
Schipperus \cite{Schipperus} in the 2000's. The next rainbow Ramsey topological theorem is due to Abraham, Cummings and Smyth \cite{AbCumSmythePolychromatic} where they showed it is consistent relative to the existence of suitable large cardinals that $\omega_2\to^*(\omega_1-cl)^2_{<\omega-bdd}$ holds. Furthermore, they deduce it from a well-known forcing axiom, Martin's Maximum \cite{ForemanMagidorShelahMM}.
Later, Garti and Zhang relaxed the boundedness requirement and showed that it is consistent relative to existence of suitable large cardinals that $\omega_2\to^* (\omega_1-st)^2_{\omega-bdd}$. There is already a notable contrast with the corresponding Ramsey theorem: By \cite{Steppingup}, for any regular $\kappa\geq \omega_1$, it is true that $\kappa\not\to [\omega_1-st]^2_\omega$, namely, there is a coloring $c: [\kappa]^2\to \omega$ such that for any $A\subseteq \kappa$ of order type $\omega_1$ that is stationary in $\sup A$, we have that $c''[A]^2 =\omega$. A natural question is whether it is consistent to have a joint extension of the two theorems mentioned above: 

\begin{myque}[Question 5.5, \cite{GartiZhang}]
		Is $\omega_2\to^*(\omega_1-cl)_{\omega\text{-bdd}}^2$ consistent?
	\end{myque}
	
The main result of this paper is a positive answer to this question. 

\begin{mysen}\label{theorem: main}
It is consistent relative to the existence of a supercompact cardinal that $\omega_2\to^* (\omega_1-cl)^2_{\omega-bdd}$ holds.
\end{mysen}

The large cardinal hypothesis we use is most likely an overkill but the conclusion does have non-trivial large cardinal strength (for example, by \cite[Remark 3.2]{GartiZhang}, it implies the consistency of a Mahlo cardinal). However, since we are not able to obtain an equiconsistency result, we do not attempt to optimize the large cardinal upper bound. See Section \ref{section: questions} for further discussions.

The methods employed in the previous theorems \cite{AbCumSmythePolychromatic, GartiZhang} depend directly on the forcing in \cite{TodorcevicPositivePartition}, which is applicable to $<\omega$-bounded colorings on $\omega_1$. The method in \cite{AbCumSmythePolychromatic} falls short of dealing with colorings that are $\omega$-bounded and the method in \cite{GartiZhang} relied on a stationary version of Chang's conjecture, which does not seem to be compatible with adding club subsets to a rainbow set in an iterable fashion.

One key new idea in our solution is to isolate a combinatorial condition, phrased game theoretically, such that the higher version of Todorcevic's forcing from \cite{TodorcevicPositivePartition} behaves nicely, such as preserving the first two uncountable cardinals. It is not clear whether our method can be used to show that $\omega_2\to^*(\omega_2)_{2-bdd}^2$ is consistent, as asked in \cite{AbCumSmythePolychromatic}. The main challenge comes from the need to iterate such a forcing many times preserving cardinals.

The new insights also allow us to improve the main theorem in \cite{GartiZhang} by reducing the large cardinal upper bound considerably (previously a huge cardinal was used): 

\begin{mysen}\label{theorem: main2}
It is consistent relative to the existence of a regular cardinal that is a stationary limit of weakly compact cardinals that $\omega_2\to^* (\omega_1-st)^2_{\omega-bdd}$ holds.
\end{mysen}

The organization of the paper is as follows: 
	\begin{enumerate}
	\item Section \ref{section: limitations} collects a few limiting results putting constraints on the ground model in order for a nice forcing to add a rainbow subset with respect to a coloring in the ground model to exist. 
	\item Section \ref{Section: Adding Rainbow Sets} contains the description of the forcing that adds a partition of $\omega_2$ into $\omega_1$ many rainbow subsets for a given $\omega$-bounded coloring on $\omega_2$ and the proof that it preserves $\omega_1$ and $\omega_2$ conditioned on a game-theoretic hypothesis.
	\item Section \ref{section: semiproper}  discusses variants of semiproperness and how ideals whose quotients satisfy these properties imply our game-theoretic hypothesis.
	\item Section \ref{section: dagger} generalizes the classical $(\dagger)$ principle to the context involving countable towers.
	\item Section \ref{Section: Weak Ideal Game} establishes the equiconsistency between the local version of our game-theoretic hypothesis and a weakly compact cardinal.
	\item Section \ref{section: adding clubs} investigates the forcing that shoots a club into a stationary subset of $\omega_2\cap \cof(\omega)$, including the proof that such forcing is $<\omega_1$-stationary set preserving and if a suitable variant of the strong Chang's conjecture holds, then it is $<\omega_1$-semiproper.
	\item Section \ref{section: maintheorem} contains the proof of Theorem \ref{theorem: main} and Theorem \ref{theorem: main2}.
	\item Section \ref{section: questions} concludes with some open questions.
	\end{enumerate}
	
We finish the introduction by fixing some notations and including a useful fact.
For $x,y$ either ordinals or sets of ordinals, we write $x<y$ if every element of $x$ (or $x$ itself) is smaller than any element of $y$ (or $y$ itself).

\begin{mydef}
		Let $\mu$ be a cardinal and $f\colon[\mu]^2\to\mu$ be a coloring. We say that $f$ is \emph{normal} if for all $a,b\in[\mu]^2$, if $f(a)=f(b)$, then $\max(a)=\max(b)$.
	\end{mydef}
	
	One useful lemma we need is: 
	
	\begin{mylem}[Lemma 1, \cite{AbCumSmythePolychromatic}]\label{Lemma: Decomposition into Normal}
		Let $\mu$ be an infinite cardinal and let $f\colon[\mu^+]^2\to\mu^+$ be a $\mu$-bounded coloring. Then there exists a club $C\subseteq \mu^+$ such that $f\uhr [C]^2$ is normal.
	\end{mylem}
	
%
%

	We may assume for the rest of the paper that all colorings being considered are normal.
	
	\section{Limitations}\label{section: limitations}
	
	In this section, we present some restrictions on the properties of a forcing that can add a rainbow subset to a given coloring. We focus on special cases concerning $\omega_1$ and $\omega_2$ for simplicity even though our arguments generalize to other successors of regular cardinals.

	\begin{mysen}\label{theoerem: cccomega2}
		If $2^{\omega_1}=\omega_2$, then there exists a 2-bounded coloring on $\omega_2$ such that no c.c.c forcing can add a rainbow subset of size $\omega_2$. 
    \end{mysen}	
    
    \begin{proof}
		Let us call a countable set consisting of pairwise disjoint pairs a \emph{block}. Namely, $A$ is a block if $A=\{\{\alpha_i,\beta_i\}: i\in \omega\}$ and $\{\alpha_i, \beta_i\}\cap \{\alpha_j, \beta_j\}=\emptyset$ whenever $i\neq j$. 
    
		As $2^{\omega_1}=\omega_2$, we can fix an enumeration $\langle \mathcal{A}_\delta: \delta\in\omega_2-\omega_1 \rangle$ such that 
		\begin{itemize}
			\item Each $\mathcal{A}_\delta = \{A^\delta_{\gamma}: \gamma\in \omega_1\}$ is an uncountable collection of blocks consisting of ordinals less than $\delta$, and 
			\item for any $\gamma\neq \gamma'$, $(\bigcup A^\delta_{\gamma}) \cap (\bigcup A^\delta_{\gamma'})=\emptyset$.
		\end{itemize}
	
		Define a 2-bounded normal coloring $c: [\omega_2]^2\to \omega_2$ satisfying that for any $\delta\in \omega_2- \omega_1$ and for any $\delta'\leq \delta$, there is $\gamma\in \omega_1$ such that for all $\{\alpha,\beta\}\in A^{\delta'}_\gamma$, $c(\alpha,\delta)=c(\beta,\delta)$.
	 
		Let us show that such construction is possible.	At stage $\delta\in \omega_2-\omega_1$, let $f: \omega_1\to \delta+1$ be a bijection. We need to define $c(\cdot, \delta)$ satisfying the requirements imposed by $\langle \mathcal{A}_{f(i)}: i\in \omega_1\rangle$. At stage $0$, we just pick $A^{f(0)}_{0}$ and define $c(\alpha,\delta)=c(\beta,\delta)$ for all $\{\alpha,\beta\}\in A^{\delta}_{f(0)}$. Suppose we are at stage $j\in \omega_1$. By the induction hypothesis, for each $i<j$, we have chosen $\eta_i\in \omega_1$ and defined $c(\cdot, \delta)$ on $\bigcup A^{f(i)}_{\eta_i}$. Let $C=\bigcup_{i<j}\bigcup A^{f(i)}_{\eta_i}$. Since $C$ is countable and $\mathcal{A}_{f(j)}$ is an uncountable collection of pairwise disjoint blocks, we can find some $\eta_j\in \omega_1$ such that $\bigcup A^{f(j)}_{\eta_j}$ is disjoint from $C$. We can then define $c(\alpha,\delta)=c(\beta,\delta)$ for all $\{\alpha,\beta\}\in A^{f(j)}_{\eta_j}$. This completes the construction.

		Finally, let us show that $c$ is as desired.
		Suppose for the sake of contradiction that there is a c.c.c. forcing $\dP$ and a nice $\dP$-name $\dot{A}$ for a subset of $\omega_2$ such that $\Vdash_{\dP}``\dot{A}$ is a $c$-rainbow subset of size $\aleph_2$". Recursively define an increasing sequence of blocks of length $\omega_1$, $\mathcal{A}=\langle A_i: i\in \omega_1\rangle$, 
	
		\begin{itemize}
			\item for $i<j\in \omega_1$, $\sup \bigcup A_i < \min \bigcup A_j$,
			\item for each $i\in \omega_1$, there exists a maximal antichain $M_i \subseteq \dP$ such that for each $r\in M_i$, there is $\{\alpha,\beta\}\in A_i$ such that $r\Vdash_{\dP} \{\alpha,\beta\}\subseteq \dot{A}$.
		\end{itemize}
	
		The construction is straightforward since $\dP$ is c.c.c. There is some $\delta\in \omega_2$ large enough such that $\{A_i: i\in \omega_1\}$ is enumerated as $\mathcal{A}_\delta$.	Find $p\in \dP$ and some $\delta'\geq \delta$ such that $p\Vdash \delta'\in \dot{A}$. By the construction of $c$, there is $i\in \omega_1$ such that for all $\{\alpha,\beta\}\in A_i$, $c(\alpha,\delta')=c(\beta,\delta')$. Let $M_i\subseteq \dP$ be the maximal antichain corresponding to $A_i$. Find $q\in M_i$ such that $q$ is compatible with $p$. In particular, there is $\{\alpha,\beta\}\in A_i$ such that $q\Vdash \{\alpha,\beta\}\subseteq \dot{A}$. A common extension of $p$ and $q$ in $\dP$ forces that $\{\alpha,\beta, \delta'\}\subseteq\dot{A}$ and $c(\alpha,\delta')=c(\beta,\delta')$, which is a contradiction to the fact that $\dot{A}$ is forced to be $c$-rainbow.
	
	\end{proof}

    \begin{mydef}
		A forcing $\dP$ is strongly proper for a class of models $\mathcal{B}$ if for all $M \in \mathcal{B}$, for any $p\in M\cap \dP$, there exists $q\leq p$ that is \emph{strongly} $(M,\dP)$-generic, which means: for any $q'\leq q$, there exists a condition $q'\restriction M \in M$ such that any $t\leq q'\restriction M$ with $t\in M$ is compatible with $q'$.
	\end{mydef}

	\begin{mysen}
		Assume $2^{\omega_1}=\omega_2$. There exists a 2-bounded coloring $c: [\omega_2]^2\to \omega_2$ such that whenever $\dP$ is strongly proper for a stationary set $S\subseteq [H(\lambda)]^{\aleph_1}$, where $\lambda= (2^{|\dP|})^+$, $\dP$ does not add a $c$-rainbow subset of size $\aleph_2$.
	\end{mysen}

	\begin{proof}
		Enumerate all possible pairs $\langle (\dP_i, \dot{A}_i): i<\omega_2\rangle$ where $\dP_i$ is a forcing of size $\aleph_1$ and $\Vdash_{\dP_i}\dot{A}_i\in [i]^{\omega_1}$. Up to isomorphism, there are indeed $\aleph_2$ many such pairs. 

		Let us define a $2$-bounded coloring $c: [\omega_2]^2\to \omega_2$ as follows. At stage $\beta\in \omega_2$, let us fix an enumeration $f: \omega_1 \to \bigcup_{\gamma<\beta} \{\gamma\}\times \dP_\gamma$. We proceed to define $c(\cdot, \beta)$ recursively in $\omega_1$ many steps. At stage $0$, let $f(0)=(\gamma_0, p_0)$, then we find $q\leq p_0$ in $\dP_{\gamma_0}$ and some $\{\alpha_0,\mu_0\}$ such that $q\Vdash_{\dP_{\gamma_0}} \{\alpha_0, \mu_0\}\subseteq \dot{A}_{\gamma_0}$. We make sure $c(\alpha_0, \beta)=c(\mu_0,\beta)$. At stage $\eta<\omega_1$, suppose we have already defined $c$ on $C=\{\alpha_i,\mu_i: i<\eta\}$ and let $f(\eta)=(\gamma_\eta, p_\eta)$. Since $\Vdash_{\dP_{\gamma_\eta}} \dot{A}_{\gamma_\eta}$ is uncountable, we can find $q\leq p_{\gamma_\eta}$ in $\dP_{\gamma_\eta}$ and $\{\alpha_\eta, \mu_\eta\}$ disjoint from $C$ such that $q\Vdash_{\dP_{\gamma_\eta}} \{\alpha_\eta, \mu_\eta\} \subseteq \dot{A}_{\gamma_\eta}$. Then we extend the definition of $c(\cdot, \beta)$ to make sure $c(\alpha_{\eta}, \beta)=c(\mu_{\eta}, \beta)$. Define $c$ injectively on other points below $\beta$. This finishes the definition of $c$, which is easily seen to be $2$-bounded.

		Let us verify that $c$ has the property as desired. Suppose for the sake of contradiction that $\dP$ is a strongly proper forcing with respect to a stationary subset $S$ of $[H(\lambda)]^{\aleph_1}$ where $\lambda=(2^{|\dP|})^+$ and $\dot{A}$ is forced to be a $c$-rainbow subset of $\omega_2$ of size $\aleph_2$. We may find $N\in S$ containing relevant objects including $\dP, \dot{A}$ and $c$ with $N\cap \omega_2=\delta$ and look at $\dP\cap N$ and $\dot{A}\cap N$. Note that $\Vdash_{\dP\cap N} \dot{A}\cap N$ is uncountable. To see this, suppose for the sake of contradiction that there is some $t\in \dP\cap N$ forcing over $\dP\cap N$ that $\dot{A}\cap N$ is countable. Let $q\leq_{\dP} t$ be a strongly $(N,\dP)$-generic condition and let $G\subseteq \dP$ be generic over $V$. As a result, $G\cap N$ is $\dP\cap N$-generic over $V$. Therefore, in $V[G]$, $(\dot{A})^G \cap N[G]$ is countable. However, $(\dot{A})^G\in N[G]\prec H(\lambda)[G]$ and $H(\lambda)[G] \models (\dot{A})^G$ has size $\aleph_2$. Hence, $(\dot{A})^G \cap N[G]$ must be uncountable. This is a contradiction.
		
		As a result, $(\dP\cap N, \dot{A}\cap N)\simeq (\dP_i, \dot{A}_i)$ for some $i\in \omega_2$. 

		Fix a witnessing ismorphism $\pi: \dP\cap N\to \dP_i$. Let $r$ be a strongly $(N,\dP)$-generic condition and $\beta>\max\{\delta,i\}$ such that $r\Vdash_{\dP} \beta\in \dot{A}-\delta$. We know that $r'=r\restriction N \in \dP\cap N$. But in our recursive construction at stage $\beta$, we make sure there are some $t\in \dP_i$ extending $\pi(r')$ in $\dP_i$ and $\{\alpha,\mu\}$ such that $t\Vdash_{\dP_i}\alpha, \mu\in \dot{A}_i$ such that $c(\alpha, \beta)=c(\mu,\beta)$. Since $\pi$ is an isomorphism, we have that $\pi^{-1}(t)$ extends $r'$ and $\pi^{-1}(t) \Vdash_{\dP\cap N} \alpha,\mu \in \dot{A}\cap N$. Let $w\in \dP$ be a lower bound for $\pi^{-1}(t)$ and $r$. Then $w \Vdash_{\dP}\alpha, \mu, \beta\in \dot{A}$ and $c(\alpha,\beta)=c(\mu,\beta)$, contradicting with the fact that $\Vdash_{\dP}\dot{A}$ is a $c$-rainbow subset.
	\end{proof}

	Abraham, Cummings and Smyth \cite{AbCumSmythePolychromatic} showed that it is consistent that there exists a 2-bounded coloring on $\omega_1$ such that no c.c.c forcing can add an uncountable rainbow subset. The Continuum hypothesis ($\CH$) fails in the model they produced. They asked if $\CH$ is compatible with the existence of such a ``c.c.c-indestructibly bad" coloring. We provide a partial positive answer.

	\begin{mysen}\label{theorem: cccpartition}
		$\CH$ implies the existence of a 2-bounded coloring $c: [\omega_1]^2\to \omega_1$ such that no c.c.c forcing can add a partition of $\omega_1$ into countably many $c$-rainbow sets.
	\end{mysen}

	\begin{proof}
		The proof is similar to that of Theorem \ref{theoerem: cccomega2}. Given a ordinal $\delta$, \emph{an increasing block of $[\delta]^2$} is of the form $\{b_i\in [\delta]^2: i<\gamma\}$ for some $\gamma<\omega_1$ such that $i<j<\gamma$ implies that $b_i<b_j$. Let $I_\delta$ be the collection of all increasing blocks of $[\delta]^2$. Note that $|I_\delta|=2^\omega= \omega_1$ and for any $\delta_0<\delta_1$, $I_{\delta_0}\subseteq I_{\delta_1}$. Let $I=\bigcup_{\delta\in \omega_1} I_\delta$.

		Let us fix enumerations $\langle \mathcal{A}_i: i<\omega_1\rangle$ and $\langle \delta_i\in \acc(\omega_1): i\in \omega\rangle$ such that
		\begin{itemize}
			\item each $\mathcal{A}_i = \{A^i_j: j\in \omega\}$ where $A^i_j \in	[I_{\delta_i}]^{\aleph_0}$ satisfies that 
			\begin{itemize}
				\item for $a\neq b\in A^i_j$, either $\bigcup a < \bigcup b$ or $\bigcup b < \bigcup a$, and
				\item for any $\xi<\delta_i$, there exists $a\in A^i_j$ with $\min \bigcup a > \xi$;
			\end{itemize}

			\item for any $\mathcal{A}=\langle A_j: j\in \omega\rangle$ such that each $A_j = \langle a_{j,k}\in I: k\in \omega_1\rangle$ satisfies $\sup \bigcup a_{j,k} < \min \bigcup a_{j,k'}$ whenever $k<k'$ and $\bigcup \bigcup A_j \cap \bigcup\bigcup A_{j'}=\emptyset$ whenever $j\neq j'$, there exist club many $\delta\in \omega_1$ for which there are unboundedly many $i\in \omega_1$ such that $\delta_i=\delta$ and $\mathcal{A}\restriction \delta =\mathcal{A}_i$.
		\end{itemize}
		Such an enumeration can be easily defined using $\CH$. Now we define a 2-bounded coloring $c: [\omega_1]^2\to \omega_1$ such that for any $i\in \omega_1$, if $\delta_i\leq i$, then we make sure that for any $j\in \omega$, there is $a\in A^i_j$ such that for all $\{\alpha,\beta\}\in a$, $c(\alpha, i)= c(\beta, i)$. To see why this is possible, recursively assume $j\in \omega$ is given and for all $k<j$, we have already picked $a_k\in A^{i}_k$ and define $c(\cdot, i)$ on $\bigcup_{k<j} \bigcup a_k$ such that for all $\{\alpha,\beta\}\in a_k$, $c(\alpha,i)=c(\beta, i)$. Consider $\xi = \max_{k<j} \sup\bigcup a_k \in \delta_i$. By our assumption, there exists $a_j\in A^i_j$ such that $\min \bigcup a_j > \xi$. As a result, we have not defined $c(\cdot, i)$ on $\bigcup a_j$. We can then define $c(\cdot, i)$ on $\bigcup a_j$ in such a way that for any $\{\alpha,\beta\}\in a_j$, $c(\alpha,i)=c(\beta,i)$. Finally, it is easy to extend the definition of $c(\cdot, i)$ from $\bigcup_{j\in \omega} \bigcup a_j$ to the whole of $i$ satisfying the $2$-bounded requirement.

	Let us verify that the coloring is as desired. Suppose for the sake of contradiction that there is a c.c.c forcing $\dP$ that adds a partition $\dot{f}: \omega_1\to \omega$ such that $\dot{f}^{-1}(n)$ is $c$-rainbow for any $n\in \omega$. We may also assume that $\Vdash_{\dP} \dot{f}^{-1}(n)$ is uncountable for all $n\in \omega$, since we can always remove some bounded set otherwise. Recursively build $\mathcal{A}=\langle A_j: j\in \omega\rangle$ as follows
		 
		\begin{itemize}
			\item $\langle A_j: j\in \omega\rangle$ such that each $A_j = \langle a_{j,k}\in I: k\in \omega_1\rangle$ satisfies $\sup \bigcup a_{j,k} < \min \bigcup a_{j,k'}$ whenever $k<k'$ and $\bigcup \bigcup A_j \cap \bigcup\bigcup A_{j'}=\emptyset$ whenever $j\neq j'$ and 
			\item for each $j\in \omega$ and $k\in \omega_1$, there exists a maximal antichain $M_{j,k}\subseteq \dP$ such that for any $r\in M_{j,k}$, there exists $b=\{\alpha,\beta\}\in a_{j,k}$ such that $r\Vdash_{\dP} b \subseteq \dot{f}^{-1}(j)$.
		\end{itemize}
		
		To see that the construction is possible, let us assume that we are at stage $k\in \omega_1$ and $j\in \omega$ and we have already defined $\langle a_{j,k'}: k'<k, j\in \omega\rangle$ and $\langle a_{j', k}: j'<j\rangle$ satisfying the requirement. Let $\xi= \sup_{(k'<k \wedge j'\in \omega) \vee (k'=k \wedge j'<j) }a_{j',k'} +1$. Since $\Vdash_{\dP} \dot{f}^{-1}(j)$ is uncountable, we can find $p_0\in \dP$ and some $b_0\in [\omega_1-\xi]^2$ such that $p_0\Vdash b_0\subseteq \dot{f}^{-1}(j)$. Suppose for some $\gamma<\omega_1$, we have defined $\langle p_l: l<\gamma\rangle$ and $\langle b_{l}\in [\omega_1]^2: l\in \gamma\rangle$ such that 	
		
		\begin{itemize}
			\item $\langle p_l: l<\gamma\rangle$ is an antichain in $\dP$, 
			\item $p_l\Vdash b_l\subseteq \dot{f}^{-1}(j)$, and 
			\item $b_l < b_{l'}$ whenever $l<l'$.
		\end{itemize}
		
		If $\langle p_l: l<\gamma\rangle$ is already a maximal antichain, we halt and let $a_{j,k}$ be $\{b_l: l<\gamma\}$. Otherwise, we find some $p_\gamma$ that is incompatible with any element in $\{p_l: l<\gamma\}$ and $b_\gamma > \sup_{l<\gamma} b_l$ such that $p_\gamma\Vdash b_\gamma \subseteq \dot{f}^{-1}(j)$. Since $\dP$ is c.c.c, this process must halt at some stage $\eta<\omega_1$. Define $a_{j,k}= \{b_l: i<\eta\}$.

		By the assumption of the enumeration $\langle\mathcal{A}_i: i\in \omega_1\rangle$ and the definition of $c$, there is some large enough $i\in \omega_1$ such that for all $j\in \omega$, there is $d_j\in A_j \restriction i$ such that for all $\{\alpha,\beta\}\in d_j$, it is the case that $c(\alpha,i)=c(\beta, i)$. Find $p\in \dP$ and $m\in \omega$ such that $p\Vdash \dot{f}(i)=m$. Let $M\subseteq \dP$ be a maximal antichain associated with $d_m$. We can find some $r\in M$ compatible with $p$. By the property of $\mathcal{A}_m$, there is $\{\alpha,\beta\}\in d_m$ such that $r\Vdash \{\alpha,\beta\}\subseteq \dot{f}^{-1}(m)$. A common extension of $p$ and $r$ forces $\{\alpha,\beta, i\}\subseteq \dot{f}^{-1}(m)$ and $c(\alpha,i)=c(\beta, i)$, contradicting the fact that $\Vdash_{\dP} \dot{f}^{-1}(m)$ is $c$-rainbow.
	\end{proof}

	\begin{mybem}
		Recall that Todorcevic showed that for any $<\omega$-bounded coloring $c: [\omega_1]^2\to \omega_1$, in the forcing extension by $\Add(\omega, \omega_1)* \mathbb{J}_{\omega_1}$ where $\mathbb{J}_{\omega_1}$ is Jensen's countably closed forcing that adds a fast club, there is a c.c.c forcing adding a countable partition of $\omega_1$ into $c$-rainbow subsets.  Theorem \ref{theorem: cccpartition} shows that in a sense, $\Add(\omega, \omega_1)$ is necessary. Compare this with the task to add an isomorphism between two $\aleph_1$-dense subsets of the reals \cite{BaumgartnerReals} where $\CH$ is enough to construct a ``fast enough" club that can be used to define a c.c.c forcing to accomplish the task.  
	\end{mybem}

%
%

	Now we recall a boundedness condition that was considered in \cite{GartiZhang}.

	\begin{mydef}[Definition 2.1, \cite{GartiZhang}]\label{deftbd}
		Let $\kappa$ and $\alpha$ be two ordinals.
		A function $f:[\kappa]^2\rightarrow\kappa$ is $<\alpha$-\emph{type bounded} iff there is an ordinal $\gamma<\alpha$ so that ${\rm otp}(t_{\alpha\beta})\leq\gamma$ whenever $\alpha<\beta<\kappa$, where $t_{\alpha\beta}=\{\xi\in\beta:f(\xi,\beta)=f(\alpha,\beta)\}$.
	\end{mydef}

	Recall the fact that for any $<\omega_1$-type-bounded $c: [\omega_1]^2\to \omega_1$, there exists a proper forcing that adds a rainbow subset by \cite{TodorcevicPositivePartition} (see also \cite{GartiZhang}). 
	
	However, the analogous statement for $\omega_2$ fails badly.

	\begin{mypro}\label{proposition: typebounded}
		For any cardinal $\kappa\geq \omega$, $\kappa^+\not\to^*(\omega_1+1)^2_{<\kappa^\omega-t-bdd}$
	\end{mypro}

	\begin{proof}
		We use a theorem due to Milner and Rado \cite{MilnerRado} (the so-called \emph{Milner-Rado paradox}): For each $\alpha<\kappa^+$, there are disjoint $\langle A^\alpha_n: n\in \omega\rangle$ such that $\otp(A^\alpha_n)\leq \kappa^n$ and $\bigcup_{n\in \omega} A^\alpha_n = \alpha$. The conclusion is clearly true when $\kappa=\omega$. Let us assume $\kappa>\omega$. Define $c: [\kappa^+]^2\to \kappa$ such that for any $\alpha<\kappa^+$ and $\beta_0,\beta_1\in \alpha$, $c(\beta_0,\alpha)=c(\beta_1,\alpha)$ iff there is $n\in \omega$, $\beta_0,\beta_1\in A^\alpha_n$. It is clear that $c$ is $<\kappa^\omega$-type-bounded. Suppose for the sake of contradiction that $A \cup \{\delta\} \subseteq \kappa^+$ is a $c$-rainbow subset of order type $\omega_1+1$. For each $\eta\in A$, we can find a unique $n\in \omega$ such that $\eta\in A^\delta_n$. By the pigeonhole principle, we can find an uncountable $A'\subseteq A$ and $n\in \omega$, such that $A'\subseteq A^\delta_n$. By the definition of $c$, $c(\cdot, \delta)\restriction A'$ is constant, which contradicts with the assumption that $A\cup \{\delta\}$ is $c$-rainbow.
	\end{proof}

	By Proposition \ref{proposition: typebounded}, there is a $<\omega_1^\omega$-type-bounded coloring $[\omega_2]^2\to \omega_2$ such that no forcing that preserves both $\omega_1$ and $\omega_2$ can add a rainbow subset of size $\omega_2$. This result will be used later to compare results we obtain in this paper with those obtained in \cite{GartiZhang}.

	\section{Adding Rainbow Sets}\label{Section: Adding Rainbow Sets}
	
	Todorcevic showed in \cite{TodorcevicPositivePartition} that, given a $<\omega$-bounded coloring $c$ on $\omega_1$, there is a proper forcing which adds partition of $\omega_1$ into countably many $c$-rainbow subsets of $\omega_1$. His poset is a three-step iteration $\Add(\omega,\omega_1)*\dJ_{\omega_1}*\dP_f$, where $\dJ_{\omega_1}$ is Jensen's poset for adding a fast club and $\dP_f$ is the poset of finite $f$-rainbow sets separated by points in the generic club. However, his proof, at least seemingly, is very specific to the case of $\omega_1$, since it requires the use of backwards induction, which only works when the conditions are finite.
	
	In \cite{AbCumSmythePolychromatic}, Abraham, Cummings and Smyth give a slightly different proof of Todorcevic's result. They replace his three-step iteration by a single side condition poset and prove the properness by using a repeated application of Fodor's lemma. Again, this proof is quite specific to the case of $\omega_1$, since a repeated shrinking of stationary subsets of $\omega_1$ appeared in the proof, which is not suitable to more than finitely many rounds. However, with help from additional assumptions, the method can be adapted to yield similar results for larger cardinals.
	
	For the rest of this section, fix a normal $\omega$-bounded coloring $f\colon[\omega_2]^2\to\omega_2$.
	
	\begin{mydef}
		Let $l\colon\omega_2\to\omega_1$ be a partial function. We say that $l$ is \emph{good for $f$} if for every $j<\omega_1$, $l^{-1}[\{j\}]$ is $f$-rainbow.
	\end{mydef}
	
	The main idea in the proof of properness comes from the possibility of amalgamating conditions, provided by the following definition:
	
	\begin{mydef}
		Let $X\subseteq\omega_2$ be countable. Define
		$$F_f(X):=\{\alpha\in\omega_2\;|\;\exists b=_{\mathrm{def}}\{\beta,\gamma\}_{<}\in[X]^2\;f(\{\alpha,\gamma\})=f(b)\}$$
	\end{mydef}
	
	Clearly, since $f$ is $\omega$-bounded, $F_f(X)$ is countable for every countable $X\subseteq\omega_2$. Since $f$ is fixed, we just write $F(X)$. The relevance of $F(X)$ is as follows:
	
	\begin{mylem}[Lemma 3, \cite{AbCumSmythePolychromatic}]\label{Lemma: Amalgamating Rainbow Sets}
		Let $X<\alpha<Y$ be such that $X\cup\{\alpha\}$ and $X\cup Y$ are both $f$-rainbow. If $\alpha\notin F(X\cup Y)$, then $X\cup\{\alpha\}\cup Y$ is $f$-rainbow.
	\end{mylem}
	
	\begin{proof}
		Suppose toward a contradiction that $X\cup\{\alpha\}\cup Y$ is not $f$-rainbow, witnessed by $\alpha_0<\beta_0$ and $\alpha_1<\beta_1$. Since $f$ is normal, $\beta_0=\beta_1$. Furthermore, $\beta_0$ is necessarily an element of $Y$, since $X\cup\{\alpha\}$ is $f$-rainbow and $X\cup\{\alpha\}<Y$. Lastly, either $\alpha_0$ or $\alpha_1$ equals $\alpha$, since $X\cup Y$ is $f$-rainbow, suppose without loss of generality that this is $\alpha_0$. But then $\beta_0$ and $\{\alpha_1,\beta_1\}$ witness that $\alpha_0=\alpha\in F(X\cup Y)$, a contradiction.
	\end{proof}
	
	Now we can define our forcing notion, which is a straightforward adaptation of the poset defined by Abraham-Cummings-Smyth \cite{AbCumSmythePolychromatic}:
	
	\begin{mydef}
		The forcing notion $\dP_f$ consists of pairs $(\mathcal{M},l)$ as follows:
		\begin{enumerate}
			\item $\mathcal{M}$ is a countable set of $\omega_1$-sized, countably closed elementary substructures of $H(\omega_4)$ containing $f$ which is totally ordered by $\in$.
			\item $l$ is a countable partial function from $\omega_2$ into $\omega_1$ which is good for $f$. Moreover, whenever $\alpha,\beta<\omega_2$ are such that $l(\alpha)=l(\beta)$, there is $M\in\mathcal{M}$ such that $\alpha\in M$ and $\beta\notin M$.
		\end{enumerate}
		For $p_0=(\mathcal{M}_0,l_0)$, $p_1=(\mathcal{M}_1,l_1)$ in $\dP_f$, we let $p_1\leq p_0$ if
		\begin{enumerate}
			\item $\mathcal{M}_0\subseteq\mathcal{M}_1$;
			\item $l_0\subseteq l_1$.
		\end{enumerate}
	\end{mydef}
	
	The following two lemmas are easy:
	
	\begin{mylem}
		$\dP_f$ is countably closed. More specifically, whenever $(\mathcal{M}_n,l_n)_{n\in\omega}$ is a descending sequence of elements of $\dP_f$, letting
		$$\mathcal{M}_{\omega}:=\bigcup_{n\in\omega}\mathcal{M}_n\text{ and }l_{\omega}:=\bigcup_{n\in\omega}l_n$$
		$p_{\omega}:=(\mathcal{M}_{\omega},l_{\omega})$ is a lower bound of $(\mathcal{M}_n,l_n)_{n\in\omega}$.
	\end{mylem}
	
	\begin{mylem}
		Let $G$ be $\dP_f$-generic. In $V[G]$, define $l_G:=\bigcup\{l\;|\;(\mathcal{M},l)\in G\}$. Then $l_G\colon\omega_2^V\to\omega_1$ is a total function which is good for $f$. Hence, in $V[G]$, $\{l_G^{-1}[\{j\}]\;|\;j\in\omega_1\}$ is a partition of $\omega_2^V$ into $\omega_1$ many $f$-rainbow sets.
	\end{mylem}
We need to show that $\omega_2^V$ is preserved as as cardinal. However, in our case, it is unclear whether $\dP_f$ is always proper for models of size $\omega_1$. To show that it is consistently $\aleph_1$-proper, we use ideals with the following property:
	
	\begin{mydef}
		Let $I$ be a ${<}\,\omega_2$-complete ideal on $\omega_2$, let $\alpha<\omega_1$ and let $\mathcal{X}$ be a collection of regressive functions from $\omega_2\to\omega_2$. We define the game $\Game^{\alpha}(I,\mathcal{X})$ as follows: The game lasts $\alpha$ many rounds with Player I starting the game. In round $\beta<\alpha$, player I plays $f_{\beta}\in\mathcal{X}$. Player II responds with an ordinal $\xi_{\beta}<\omega_2$. At the end of the game, player II wins if and only if there is a set $X\in I^+$ such that for every $\eta\in X$ and $\beta<\alpha$, $f_{\beta}(\eta)<\xi_{\beta}$.
		
		If $\mathcal{X}$ is the collection of all regressive functions from $\omega_2\to\omega_2$, we let $\Game^{\alpha}(I,\mathcal{X}):=\Game^{\alpha}(I)$.
	\end{mydef}
	The following is an illustration of a round of $\Game^{\alpha}(I,\mathcal{X})$:
	$$	
	\begin{matrix}
\mathrm{I} &  f_0\in \mathcal{X} & \quad & f_1\in \mathcal{X} & \quad &  \cdots & f_\beta \in \mathcal{X} & \quad & \cdots \\
\mathrm{II} &  \quad & \xi_0\in \omega_2 & \quad & \xi_1\in \omega_2 & \cdots &\quad &\xi_\beta\in \omega_2 &\cdots
\end{matrix}
	$$
Player II wins iff $$\{\eta\in \omega_2: \forall \beta<\alpha, f_\beta(\eta)<\xi_\beta\}\in I^+.$$
	\linebreak	
	We will investigate this game further in Section \ref{Section: Weak Ideal Game}. For now, we simply remark that ideals where II has a winning strategy can consistently exist. For example, if $\kappa$ is measurable and $G$ is $\Coll(\omega_1,<\kappa)$-generic, Laver (see \cite{GalvinJechMagidorIdealGame}) showed that, in $V[G]$, there is a normal, ${<}\,\omega_2$-complete ideal $I$ on $\omega_2$ such that $I^+$ contains a countably closed dense subset. It is straightforward to see that for every $\alpha<\omega_1$, II has a winning strategy in $\Game^{\alpha}(I)$.
	
	For now, assume $\CH$ holds and that for every $\omega_2$-sized collection $\mathcal{X}$ of regressive functions from $\omega_2$ into $\omega_2$ and every $\alpha\in \omega_1$, there exists a ${<}\,\omega_2$-complete ideal $I_{\mathcal{X}}$ on $\omega_2 \cap \cof(\omega_1)$ such that player II wins $\Game^{\alpha}(I_{\mathcal{X}}, \mathcal{X})$. Our goal is to show the following:
	
	\begin{mypro}\label{Proposition: Properness}
		Let $\Theta$ be sufficiently large and $M\prec H(\Theta)$ an $\omega_1$-sized, countably closed elementary submodel of $H(\Theta)$ containing relevant objects, including $\dP_f$. Suppose that $p=(\mathcal{M},l)\in\dP_f$ such that $M\cap H(\omega_4)\in\mathcal{M}$. Then $p$ is $(M,\dP_f)$-generic.
	\end{mypro}
	
	The proof is based on a combinatorial lemma:
	
	\begin{mylem}\label{Lemma: Refining Sequence}
		Let $\gamma$ be a countable ordinal and let $(M_i)_{i<\gamma}$ be an $\in$-increasing sequence of $\omega_1$-sized countably closed elementary substructures of $H(\omega_4)$. Let $l_{\mathrm{low}}$, $l_{\mathrm{high}}$ be countable partial functions from $\omega_2$ to $\omega_1$ such that the following holds:
		\begin{enumerate}
			\item $f\in M_0$;
			\item $l_{\mathrm{low}}\cup l_{\mathrm{high}}$ is good for $f$;
			\item $l_{\mathrm{low}}\subseteq M_{0}$ and $l_{\mathrm{high}}\cap M_{0}=\emptyset$;
			\item whenever $\alpha,\beta\in\dom(l_{\mathrm{high}})$ are such that $l_{\mathrm{high}}(\alpha)=l_{\mathrm{high}}(\beta)$, there is $i\in\gamma$ such that $\alpha\in M_i$ and $\beta\notin M_i$.
		\end{enumerate}
		
		Let $(l_i)_{i<\omega_2}\in M_{0}$ be a sequence of countable partial functions from $\omega_2$ to $\omega_1$ such that the following holds:
		\begin{enumerate}
			\item For every $i<\omega_2$, $l_{\mathrm{low}}\cup l_i$ is good for $f$,
			\item for every $i<j<\omega_2$, $\dom(l_i)<\dom(l_j)$, and
			\item there is $\delta^*\in \omega_1$ such that for all $l\in \omega_1$, $l_i[\dom(l_i)]\subseteq \delta^*$.
		\end{enumerate}
		
		Let $\mathcal{X}$ be the set of all regressive functions on $\omega_2$ which are definable over $H(\omega_3)$ from $[\omega_2]^{\aleph_0} \cup \{(l_i)_{i<\omega_2}, f \}$. Then there is a set $X\in I_{\mathcal{X}}^+$ such that for every sufficiently large $\eta\in X$, $l_{\eta}\cup l_{\mathrm{low}}\cup l_{\mathrm{high}}$ is good for $f$.
	\end{mylem}
	
	\begin{proof}
		First of all, note that clearly $|\mathcal{X}|=\omega_2$ since $|[\omega_2]^{\aleph_0}|=\omega_2$ and so $I_{\mathcal{X}}$ actually exists. Fix a winning strategy $\sigma$ for player II in $\Game^{\gamma}(I_{\mathcal{X}},\mathcal{X})$. Note that, since $\{H(\omega_3), f, (l_i)_{i<\omega_2}, \omega_2 ,\gamma\}\in M_0$, by elementarity, $\mathcal{X}\in M_0$ (and we may assume $I_{\mathcal{X}}\in M_0$) and so $\sigma\in M_0$.
		
		Let $\beta\in\gamma$. We define $k_{\beta}:=l_{\mathrm{low}}\cup (l_{\mathrm{high}}\uhr M_{\beta})\in M_\beta$ and let $f_{\beta}$ be defined as follows: If $\cof(\eta)\neq\omega_1$, $f_{\beta}(\eta)=0$. Otherwise,
		$$f_{\beta}(\eta):=\sup_{i<\delta^*}(\eta\cap F(k_{\beta}^{-1}[\{i\}]\cup l_{\eta}^{-1}[\{i\}]))$$
		$f_{\beta}$ is a supremum of a countable collection of ordinals below $\eta$ and thus $f_{\beta}$ is a regressive function. Furthermore, $f_{\beta}$ is definable from $(l_i)_{i<\omega_2}$, $f$ and the function $k_{\beta}$ which can be coded as a countable subset of $\omega_2$, so clearly $f_{\beta}\in\mathcal{X}\cap M_{\beta}$.
		
		Let $(f_{\beta},\xi_{\beta})_{\beta<\gamma}$ be a run of $\Game^{\gamma}(I_{\mathcal{X}},\mathcal{X})$ where II plays according to $\sigma$. For every $\beta<\gamma$, $(f_{\alpha})_{\alpha\leq\beta}\in M_{\beta}$ and thus $\xi_{\beta}\in M_{\beta}$ as well. Because $\sigma$ is a winning strategy, there is a set $X\in I_{\mathcal{X}}^+$ of ordinals with cofinality $\omega_1$ such that for every $\eta\in X$ and $\beta<\gamma$, $f_{\beta}(\eta)<\xi_{\beta}$. We may assume that for every $\eta\in X$, $\eta$ and $\dom(l_{\eta})$ are larger than $\sup_{i<\gamma}M_i\cap\omega_2$. Let us show that $X$ is as required.
		
		To this end, fix $\eta\in X$. We will show by induction on $\beta<\gamma$ that $m_{\beta}:=l_{\eta}\cup k_{\beta}$ is good for $f$. For $\beta=0$, this holds by assumption. Suppose the statement holds for all $\alpha<\beta$. Let $j<\delta^*$. It follows that there is at most one $\zeta\in M_{\beta}\smallsetminus\bigcup_{\alpha<\beta}M_{\alpha}$ such that $l_{\mathrm{high}}(\zeta)=j$. If there is no such $\zeta$, $m_{\beta}^{-1}[\{j\}]=\bigcup_{\alpha<\beta}m_{\alpha}^{-1}[\{j\}]$ which is $f$-rainbow by the inductive assumption. Otherwise, we have that $\zeta>\xi_{\alpha}$ for every $\alpha<\beta$, since $\xi_{\alpha}\in M_{\alpha}$. Since $\zeta<\eta$ and $f_{\alpha}(\eta)<\xi_{\alpha}$ by assumption, this implies $\zeta\notin F(k_{\alpha}^{-1}[\{j\}]\cup l_{\eta}^{-1}[\{j\}])$. Therefore, since $F$ preserves increasing unions, we have
		\begin{enumerate}
		\item $\bigcup_{\alpha<\beta}k_{\alpha}^{-1}[\{j\}]<\zeta<l_{\eta}^{-1}[\{j\}]$, 
		\item $\bigcup_{\alpha<\beta}k_{\alpha}^{-1}[\{j\}] \cup \{\zeta\}$ is $f$-rainbow since it is a subset of $(l_{\mathrm{low}}\cup l_{\mathrm{high}})^{-1}[\{j\}]$ which is $f$-rainbow,
		\item $\bigcup_{\alpha<\beta}k_{\alpha}^{-1}[\{j\}] \cup l_{\eta}^{-1}[\{j\}]$ is $f$-rainbow by the induction hypothesis, and
		\item $\zeta\notin F\left(\left(\bigcup_{\alpha<\beta}k_{\alpha}^{-1}[\{j\}]\right)\cup l_{\eta}^{-1}[\{j\}]\right)$
		\end{enumerate}
Hence, $k_{\beta}^{-1}[\{j\}]\cup l_{\eta}^{-1}[\{j\}]$ is $f$-rainbow by Lemma \ref{Lemma: Amalgamating Rainbow Sets}.
	\end{proof}
	
	\begin{proof}[Proof of Proposition \ref{Proposition: Properness}]
		Let $D\in M$ be an open dense subset of $\dP_f$. By strengthening $p$ if necessary, we may assume that $p\in D$. Let $l_M:=l\uhr M$ and $p_M:=(\mathcal{M}\cap M,l_M)$. Note that $l_M\in M$. Also let $\delta:=M\cap\omega_2$, $\gamma:=\sup(\dom(l_M))<\delta$ and $\delta^* := \sup l[\dom(l)]+1 \in \omega_1$.
		
		Suppose for the sake of contradiction that there is no $q\in D\cap M$ which is compatible with $p$. Let $B$ consist of all pairs $(\alpha,k)$ such that
		\begin{enumerate}
			\item $\alpha\in\omega_2$;
			\item $k$ is a countable partial function from $\omega_2$ to $\delta^*$ such that $\dom(k)$ is disjoint from $\alpha$ and $k\cup l_M$ is good for $f$;
			\item there is no condition $q=(\mathcal{N},m)\in D$ such that
			\begin{enumerate}
				\item $q\leq p_M$;
				\item $\dom(m)\subseteq\alpha$ and $N\cap\omega_2<\alpha$ for every $N\in\mathcal{N}$;
				\item $m\cup k$ is good for $f$.
			\end{enumerate}
		\end{enumerate}
		
		Let $l^M:=l\smallsetminus M$.
		
		\begin{myclaim}
			$(\delta,l^M)\in B$.
		\end{myclaim}
		
		\begin{proof}
			Clearly, $\delta\in\omega_2$ and $l^M$ is a countable partial function from $\omega_2-\delta$ to $\delta^*$ such that $l_M\cup l^M=l$ is good for $f$. Suppose for the sake of contradiction that $(\delta,l^M)\not\in B$. There must be a condition $q=(\mathcal{N},m)\in D$ such that $q\leq p_M$, $\dom(m)\subseteq\delta$ and $N\cap\omega_2<\delta$ for every $N\in\mathcal{N}$ and $m\cup l^M$ is good for $f$. In particular, $\{N\cap\omega_2\;|\;N\in\mathcal{N}\}$ is an element of $M$. Since $M\prec H(\Theta)$ and $\Theta$ is sufficiently large, we may without loss of generality assume that $q\in M$. But then $(\mathcal{N}\cup\mathcal{M},m\cup l^M)$ witnesses that $q$ and $p$ are compatible, contradicting with our assumptions.
		\end{proof}
		
		So for every $\beta\in M\cap\omega_2$, there is $(\delta_\beta,l_\beta)\in B$ with $\beta<\delta_\beta$. By the elementarity of $M$, since $B\in M$, this holds for every $\beta\in\omega_2$. We can therefore construct a sequence $(\delta_{\beta},l_{\beta})_{\beta<\omega_2}$ of elements of $B$ such that for every $\beta_0<\beta_1$, $\delta_{\beta_1}>\max\{\delta_{\beta_0},\sup \dom(l_{\beta_0})\}$. By Lemma \ref{Lemma: Refining Sequence}, there is $B\subseteq\omega_2$ which is $I$-positive such that for every $\eta\in B$, $l_{\eta}\cup l$ is good for $f$. In particular, $B$ is unbounded in $\omega_2$ and therefore we can choose such an $\eta$ for which $\delta_{\eta}$ is above $\dom(l)$ and $M'\cap\omega_2$ for every $M'\in\mathcal{M}$. But then $p\in D$ itself witnesses that $(\delta_{\eta},l_{\eta})\notin B$, a contradiction.
	\end{proof}
	
	In particular, assuming $\CH$, there are stationarily many $M\prec H(\omega_4)$ which are countably closed and have size $\omega_1$. Thus, to summarize this section, we have shown the following:
	
	\begin{mysen}\label{Theorem: Adding Rainbow Set}
		Assume $\CH$ and suppose that for every $\omega_2$-sized collection $\mathcal{X}$ of regressive functions on $\omega_2$ and every $\alpha\in \omega_1$, there exists a ${<}\,\omega_2$-complete ideal $I_{\mathcal{X}}$ on $\omega_2$ such that $E_{\omega}^{\omega_2}\in I_{\mathcal{X}}$ and player II has a winning strategy in $\Game^{\alpha}(I_{\mathcal{X}}, \mathcal{X})$. Let $f\colon[\omega_2]^2\to\omega_2$ be a normal, $\omega$-bounded coloring on $\omega_2$.
		
		Then there is a poset $\dP_f$ such that the following holds:
		\begin{enumerate}
			\item $\dP_f$ is countably closed;
			\item $\dP_f$ preserves $\omega_2^V$;
			\item $\dP_f$ adds a total function $l\colon\omega_2\to\omega_1$ such that for every $i\in\omega_1$, $l^{-1}[\{i\}]$ is $f$-rainbow.
		\end{enumerate}
	\end{mysen}
	
	\section{$(\alpha,\delta)$-semiproperness}\label{section: semiproper}
	In this section, we focus on the global version of the main technical hypothesis of Theorem \ref{Theorem: Adding Rainbow Set}: the existence of a $<\omega_2$-complete ideal $I$ on $\omega_2$ such that player II wins $\Game^{\alpha}({I})$ for every countable ordinal $\alpha$. Later in Section \ref{Section: Weak Ideal Game}, we investigate the local version optimizing the large cardinal hypothesis used.
	
	In general, showing that such ideals can consistently exist is not problematic. For example, by the methods in \cite{GalvinJechMagidorIdealGame}, such a ideal exists in $V[\Coll(\omega_1,<\kappa)]$ whenever $\kappa$ is a measurable. However, in our intended model, where for every $\omega$-bounded coloring $c$ on $\omega_2$, there exists a \emph{closed} $c$-rainbow subset of order type $\omega_1$, we must combine our forcing from Theorem \ref{Theorem: Adding Rainbow Set} with a forcing notion due to Shelah adding a closed set of ordertype $\omega_1$ into a stationary subset of $E_{\omega}^{\omega_2}$. Such a poset cannot be proper, let alone countably closed.
	
	In this section, we introduce a combination of $\alpha$-properness and semiproperness which we call $(\alpha,\delta)$-semiproperness. This is  a weakening of countable closure, which is sufficient for our purpose, at the same time compatible with the club shooting forcing into a stationary subset of $E_{\omega}^{\omega_2}$.
	
	\begin{mydef}
		Let $\dP$ be a poset and $\delta$ a cardinal. We say that $\dP$ is \emph{$\delta$-semiproper} if for every large enough regular $\Theta$, every countable $M\prec H(\Theta)$ with $\dP\in M$ and every $p\in M\cap\dP$ there is a condition $q\leq p$ which is $(M,\dP)$-$\delta$-semigeneric, namely, $q\Vdash_{\dP}M[\dot{G}]\cap\delta=M\cap\delta$.
	\end{mydef}
	
	So the usual notion of \emph{semiproperness} is the same as $\omega_1$-semiproperness in this terminology. The notion of \emph{$\alpha$-properness} (\cite[Section 5.1]{AbrahamProperChapter}), which $(\alpha,\delta)$-semiproperness is modeled after, uses $\alpha$-towers of elementary substructures instead of single elementary substructures:
	
	\begin{mydef}
		Let $\alpha$ be a countable ordinal and $\overline{M}=(M_i)_{i<\alpha}$ a sequence of countable elementary substructures of $H(\Theta)$, where $\Theta$ is regular uncountable. We say that $\overline{M}$ is an \emph{$\alpha$-tower} if
		\begin{enumerate}
			\item For every limit $\delta<\alpha$, $M_{\delta}=\bigcup_{i<\delta}M_i$;
			\item For every successor $j<\alpha$,
			$$(M_i)_{i< j}\in M_j$$
		\end{enumerate}
	\end{mydef}
	
	\begin{mydef}\label{definition: (alpha,delta)-semiproper}
		Let $\dP$ be a poset, $\alpha$ a countable ordinal and $\delta$ a cardinal. We say that $\dP$ is $(\alpha,\delta)$-semiproper if for every sufficiently large $\Theta$, there is $x\in H(\Theta)$ such that for every $\alpha$-tower $(M_i)_{i<\alpha}$ of countable elementary substructures of $H(\Theta)$, the following holds: Whenever $\dP,x\in M_0$ and $p\in\dP\cap M_0$, there is $q\leq p$ such that $q$ is $(M_i,\dP)$-$\delta$-semigeneric for every $i<\alpha$.
	\end{mydef}
	
	We will call such a condition $((M_i)_{i<\alpha},\dP)$-$\delta$-semigeneric. If $\delta=\omega_1$, we may drop it.
	
	The connection between the notion of $(\alpha,\delta)$-semiproperness and our game $\Game^{\alpha}(I)$ comes from the consideration of the following game, reminiscent of the properness game:
	
	\begin{mydef}
		Let $\dP$ be a poset, $p\in\dP$, $\alpha$ a countable ordinal and $\delta$ a cardinal. The game $\Game P_{\delta}^{\alpha}(\dP,p)$ is played as follows: It lasts $\omega\cdot\alpha$ many rounds.  If $\gamma<\omega\alpha$ is even, player I plays a $\dP$-name $\dot{\xi}_{\gamma}$ for an ordinal in $\delta$. If $\gamma$ is odd, player II plays an ordinal $\xi_{\gamma}$ in $\delta$.
		
$$	
	\begin{matrix}
\mathrm{I} & \dot{\xi}_0 & \quad & \cdots & \dot{\xi}_{\omega} &  \quad & \cdots & \dot{\xi}_{\omega\cdot \beta} &  \quad & \cdots \\
\mathrm{II} &  \quad & \xi_1 & \cdots & \quad & \xi_{\omega+1} &\cdots &\quad & \xi_{\omega\cdot \beta+1} &\cdots
\end{matrix}
	$$		
		
		After $\omega\cdot\alpha$ many rounds, player II wins if there is a condition $q\leq_{\dP} p$ which forces that, for every limit ordinal $\gamma<\omega\alpha$,
		$$\{\dot{\xi}_{\gamma+n}\;|\;n\in\omega\}\subseteq\{\check{\xi}_{\gamma+n}\;|\;n\in\omega\}$$
	\end{mydef}
	
	This game was first considered by Shelah in \cite[Chapter XII]{ShelahProperImproper}. Using appropriate bookkeeping, we may allow player II to instead play a countable set of ordinals instead of a single one.
	
	\begin{mypro}\label{proposition: equivalence}
		Let $\dP$ be a poset, $\alpha$ be a countable ordinal and $\delta$ be a cardinal. The following are equivalent:
		\begin{enumerate}
			\item $\dP$ is $(\alpha,\delta)$-semiproper;
			\item for every $p\in\dP$, player II has a winning strategy in $\Game P_{\delta}^{\alpha}(\dP,p)$.
		\end{enumerate}
	\end{mypro}
	
	\begin{proof}
		First suppose that $\dP$ is $(\alpha,\delta)$-semiproper and $p\in\dP$. We define a winning strategy for II.
		
		Assume I has played $(\dot{\xi}_{\beta})_{\beta\leq\gamma}$ so far. On the side, II has constructed a continuous sequence $(M_i)_{i<\gamma}$ of countable elementary submodels $H(\Theta)$ where $\Theta$ is a sufficiently large regular cardinal such that $p\in M_0$ and, for every successor $i<\gamma$, $(M_j)_{j<i}\in M_i$. II now finds $M_{\gamma}$ such that $p,\dot{\xi}_{\gamma},(M_j)_{j<\gamma}\in M_{\gamma}$ and plays $\delta\cap M_{\gamma}$.
		
		Now let $(N_i)_{i<\alpha}$ be defined as follows: For every $i<\alpha$, $N_i:=\bigcup_{j<\omega\cdot i}M_j$.
		
		\begin{myclaim}
			$(N_i)_{i<\alpha}$ is an $\alpha$-tower.
		\end{myclaim}
		
		\begin{proof}
			Clearly, $(N_i)_{i<\alpha}$ is continuous. Let $i<\alpha$ be a successor, $i=i'+1$. Then $(N_j)_{j<i}=(N_j)_{j\leq i'}$ can be constructed from $(M_j)_{j\leq \omega \cdot i'}$ which is an element of $M_{\omega\cdot i'+2}\subseteq N_i$.
		\end{proof}
		
		So, by assumption, $p\in\dP\cap N_0$. Thus, since $\dP$ is $(\alpha,\delta)$-semiproper, we can find $q\leq p$ which is $((N_i)_{i<\alpha},\dP)$-$\delta$-semigeneric for every $i<\alpha$. We claim that $q$ is as required. Let $\gamma=\omega\cdot \beta$ be a limit and $n<\omega$. Then $\dot{\xi}_{\gamma+n}\in M_{\gamma+\omega}=N_{\beta+1}$. Since $q$ is $(N_{\beta+1},\dP)$-$\delta$-semigeneric, $q$ forces that $\dot{\xi}_{\gamma+n}\in N_{\beta+1}[\dot{G}]\cap \delta = M_{\gamma+\omega} \cap \delta$. In particular, $q$ forces that $\dot{\xi}_{\gamma+n}$ is in $M_i$ for some $i<\omega(\beta+1)=\omega\beta+\omega=\gamma+\omega$. As a result, the strategy we defined for player II is indeed a winning strategy.
		
		Now suppose that II has a winning strategy $\sigma_p$ in $\Game P_{\delta}^{\alpha}(\dP,p)$ for every $p\in\dP$. Let $(M_i)_{i<\alpha}$ be an $\alpha$-tower and $p\in\dP\cap M_0$. By choosing $\Theta$ sufficiently large, $p,\dP\in M_0$ implies that the winning strategy $\sigma=\sigma_p$ is an element of $M_0$. We now play a game of $\Game P_{\delta}^{\alpha}(\dP,p)$ by letting I successively play, between $\omega\gamma$ and $\omega(\gamma+1)$, all names for ordinals in $\delta$ which are in $M_{\gamma+1}$, based on some fixed bijection $e_{\gamma+1}: \omega \to M_{\gamma+1}$. By the fact that $\sigma, (M_i)_{i\leq \gamma}, (e_{i})_{i\leq \gamma}\in M_{\gamma+1}$, every response $\xi_{\omega\gamma+n}$ is in $M_{\gamma+1}$. Now let $q\leq p$ witness that II wins the game. It follows that $q$ forces that, for every name $\dot{\xi}\in M_{\gamma+1}\cap \delta$, the evaluation of $\dot{\xi}$ has been played at some stage $\omega\gamma+n$ and thus lies in $M_{\gamma+1}\cap \delta$. Therefore, $q$ is $(M_i,\dP)$-$\delta$-semigeneric for every ordinal $i<\alpha$ as required.
	\end{proof}
	
	The following definition and proposition explains why the notion of \linebreak
	$(\alpha,\delta)$-semiproperness is connected to our endeavor.
	
	\begin{mydef}
		Let $\kappa$ be a cardinal and $I$ an ideal over $\kappa$. The poset $P(\kappa)/I$ consists of all equivalence classes $[A]_I$ where $A\in I^+$. The order of the poset is: $[B]_I\leq[A]_I$ iff $B\smallsetminus A\in I$.
	\end{mydef}
	
	\begin{mypro}\label{Proposition: Ideal Game}
		Let $I$ be a ${<}\,\omega_2$-complete normal ideal on $\omega_2$ and $\alpha$ be a countable ordinal. Suppose that $P(\omega_2)/I$ is $(\alpha,\omega_2)$-semiproper. Then II has a winning strategy in $\Game^{\alpha}(I)$.
	\end{mypro}
	
	\begin{proof}
		Let $G$ be any $P(\omega_2)/I$-generic filter over $V$.
		
		\begin{myclaim}
			In $V[G]$, there is a filter $U$ over $\omega_2^V$ such that:
			\begin{enumerate}
				\item For every $A\subseteq\omega_2^V$ in $V$, either $A\in U$ or $\omega_2^V\smallsetminus A\in U$;
				\item whenever $A\in I$, $\omega_2^V\smallsetminus A\in U$;
				\item whenever $f\in V$ is a regressive function on $\omega_2^V$, $f$ is constant on a set in $U$.
			\end{enumerate}
		\end{myclaim}
		
		\begin{proof}
			Let $U$ consist of all $A$ such that $[A]_I\in G$. It is easy to see that $U$ is a filter over $\omega_2^V$ and that for every $A\subseteq\omega_2^V$, either $A\in U$ or $\omega_2^V\smallsetminus A\in U$.
			
			To see (2), if $A\in I$, $\omega_2^V\smallsetminus A$ contains every subset of $\omega_2^V$ modulo $I$. Thus, $[\omega_2^V\smallsetminus A]\in G$.
			
			To see (3), if $f$ is regressive and $[A]_I\in P(\omega_2)/I$ is arbitrary, there is $B\subseteq A$ in $I^+$ such that $f$ is constant on $B$. Ergo, the set of all $[B]_I$ such that $f$ is constant on $B$ is dense in $P(\omega_2)/I$, which easily implies the desired statement.
		\end{proof}
		
		Let $f\in V$ be a regressive function on $\omega_2^V$. Then $f$ is constant on a set in $U$ with some value $\xi$. Since $U$ is a filter, there is exactly one $\xi$ such that $\{\eta\in\omega_2^V\;|\;f(\eta)=\xi\}\in U$. Let $\dot{\xi}_f$ be a name for this $\xi$.
		
		By Proposition \ref{proposition: equivalence}, we  can fix a winning strategy for player II in $\Game P^{\alpha}(P(\omega_2)/I, [\omega_2]_I)$. We define a winning strategy for player II in $\Game^{\alpha}(I)$. Suppose player I has played $(f_{\beta})_{\beta\leq\gamma}$ so far. We turn this into a partial play $(\dot{\xi}_i)_{i<\omega(\gamma+1)}$ by defining $\dot{\xi}_i:=\dot{\xi}_{f_{\beta}}$ for all $i\in[\omega\beta,\omega(\beta+1))$. Let $(\xi_i)_{i<\omega(\gamma+1)}$ be the sequence of responses of player II according to their winning strategy in $\Game P^{\alpha}_{\omega_2}(P(\omega_2)/I, [\omega_2]_I)$ and let player II play the ordinal
		$$\zeta_{\gamma}:=\sup\{\xi_i\;|\;i\in[\omega\gamma,\omega(\gamma+1))\}$$
		back in the game $\Game^{\alpha}(I)$.
		
		We show that this is a winning strategy for player II. Let $(f_{\beta},\zeta_{\beta})_{\beta<\alpha}$ be a run of the game $\Game^{\alpha}(I)$ where player II played according to the strategy. Let $(\dot{\xi}_i,\xi_i)_{i<\omega\alpha}$ be the corresponding run of $\Game P^{\alpha}_{\omega_2}(P(\omega_2)/I, [\omega_2]_I)$. Since player II played this game according to a winning strategy, there is a condition $[A]_I\in P(\omega_2)/I$ forcing that for every $\beta<\alpha$,
		$$\{\dot{\xi}_i\;|\;i\in[\omega\beta,\omega(\beta+1))\}\subseteq\{\xi_i\;|\;i\in[\omega\beta,\omega(\beta+1))\}$$
		In particular, $[A]_I$ forces that, for every $\beta<\alpha$, the function $\check{f}_{\beta}$ is constant on a set in $\dot{U}$ with value equal to some $\xi_i$ for $i\in[\omega\beta,\omega(\beta+1)))$ and thus with value bounded by $\zeta_{\beta}$.
		
		\begin{myclaim}
			For every $\beta<\alpha$,
			$$A_{\beta}:=\{\eta\in A\;|\;f_{\beta}(\eta)\geq\zeta_{\beta}\}\in I$$
		\end{myclaim}
		
		\begin{proof}
			Assume toward a contradiction that $A_{\beta}\in I^+$. In particular, $[A_{\beta}]\leq[A]$. However, if $G$ is $P(\omega_2)/I$-generic containing $[A_{\beta}]$, $A_{\beta}$ is in the filter induced by $G$. But then $f_{\beta}$ cannot be constant on a set in $U$ with value bounded by $\zeta_{\beta}$, a contradiction.
		\end{proof}
		
		Now, let
		$$A^*:=A\smallsetminus\bigcup_{\beta<\alpha}A_{\beta}$$
		Since $I$ is ${<}\,\omega_2$-complete, $\bigcup_{\beta<\alpha}A_{\beta}\in I$ and so $A^*\in I^+$. It follows that for every $\beta<\alpha$ and $\eta\in A^*$, $f_{\beta}(\eta)<\zeta_{\beta}$ as required.
	\end{proof}
	
	In our main theorem, we will not use precisely this lemma, since we need to work inside some inner model. Instead we will use the following variant. For the purposes of this paper, an \emph{iteration} $(\dB_i)_{i<\gamma}$ is defined to be any sequence of complete Boolean algebras such that for all $i<j<\gamma$, $\dB_i$ is a complete suborder of $\dB_j$.
	
	\begin{mylem}\label{Lemma: Ideal Existence}
		Let $j\colon V\to M$ be an ultrapower embedding by a normal measure on $\kappa$, $\alpha$ be a countable ordinal and $(\dB_i)_{i<\kappa+1}$ an iteration such that
		\begin{enumerate}
			\item For every $i<\kappa$, $|\dB_i|<\kappa$;
			\item $\dB_{\kappa}$ is the direct limit of $(\dB_i)_{i<\kappa}$;
			\item In $M$, $\dB_{\kappa}$ forces that $j(\dB_{\kappa})/\dB_{\kappa}$ is $(\alpha,\kappa)$-semiproper.
		\end{enumerate}
		
		Let $G$ be $\dB_{\kappa}$-generic. In $V[G]$, there exists a $<\kappa$-complete normal ideal $I$ on $\kappa$ such that player II has a winning strategy in the game $\Game^{\alpha}(I)$.
	\end{mylem}
	
	\begin{proof}
	Let $G\subseteq \dB_\kappa$ be generic over $V$. By the assumption, in $M[G]$, $j(\dB_{\kappa})/G$ is $(\alpha,\kappa)$-semiproper. We define the ideal $I$ as follows: Let $A\subseteq\kappa$, $A=\tau^G$ for some $\dB_{\kappa}$-name $\tau$. It follows that $j(\tau)$ is a $j(\dB_{\kappa})$-name. Let $j(\tau)/G$ be the corresponding $j(\dB_{\kappa})/G$-name. We let $A\in I$ if $\Vdash_{j(\dB_{\kappa})/G}\check{\kappa}\in j(\tau)/G$.
		
		We show that $I$ is as required. The fact that $I$ is normal and $<\kappa$-complete is immediate from the properties of $j$. We verify that player II has a winning strategy in the game $\Game^{\alpha}(I)$. To this end, fix a winning strategy $\sigma$ for player II in $\Game P^{\alpha}_\kappa(j(\dB_{\kappa})/G, 1_{j(\dB_{\kappa})/G})$ in $M[G]$.
		
		Let $f\colon\kappa\to\kappa$ be a regressive function in $V[G]$ and let $\tau_f$ be a $\dB_{\kappa}$-name such that $\tau_f^G=f$. It follows that $j(\tau_f)\in M$ and thus the $j(\dB_{\kappa})/G$-name for an ordinal in $\kappa$, $j(\tau_f)(\kappa)/G$ is in $M[G]$.
		
		Suppose we are in a run of the game where player I has played $(f_{\beta})_{\beta\leq\gamma}$. Let $i<\omega(\gamma+1)$ and let $\beta\leq\gamma$ be the unique ordinal such that $i=\omega\beta+n$ for $n\in\omega$. Define $\dot{\xi}_i:=j(\tau_{f_{\beta}})(\kappa)/G$. For every $i<\omega(\gamma+1)$, $\dot{\xi}_i\in M[G]$. Since $\dB_{\kappa}$ is $\kappa$-cc, $M[G]$ is closed under $\kappa$-sequences and thus $(\dot{\xi}_i)_{i<\omega(\gamma+1)}\in M[G]$. Then we proceed exactly as in the proof of Proposition \ref{Proposition: Ideal Game} in $M[G]$. To be precise, let $(\xi_i)_{i<\omega(\gamma+1)}$ be the sequence of responses of player II according to their winning strategy for the game $\Game P^{\alpha}_\kappa(j(\dB_{\kappa})/G, 1_{j(\dB_{\kappa})/G})$ in $M[G]$ and define
		$$\zeta_{\gamma}:=\sup_{n\in\omega}\xi_{\omega\gamma+n}$$
		Let II play $\zeta_{\gamma}$ in response to $(f_{\beta})_{\beta\leq\gamma}$ in $\Game^{\alpha}(I)$.
		
		We show that this constitutes a winning strategy. To this end, let $(f_{\beta})_{\beta<\alpha}$ be a run of the game where II follows the strategy described above. It follows that II wins the corresponding run $(\dot{\xi}_i,\xi_i)_{i<\omega\alpha}$ of $\Game P^{\alpha}_{\kappa}(j(\dB_{\kappa})/G,1_{j(\dB_{\kappa})/G})$ in $M[G]$. Therefore, there is a condition $q\in j(\dB_{\kappa})/G$ such that for every limit $\gamma<\alpha$,
		$$q\Vdash\{\dot{\xi}_{\omega\gamma+n}\;|\;n\in\omega\}\subseteq\{\xi_{\omega\gamma+n}\;|\;n\in\omega\}$$
		Define
		$$B:=\{\eta\in\kappa\;|\;\forall\beta<\alpha\;f_{\beta}(\eta)<\zeta_{\beta}\}$$
		Let $\tau$ be a $\dB_{\kappa}$-name such that $\tau^G=B$. By the construction, it follows that for every $\beta<\alpha$, $q$ forces that $j(\tau_{f_{\beta}})(\check{\kappa})/G$ is equal to some $\xi_{\omega\beta+n}$ and thus below $\zeta_{\beta}$. Hence, $q\Vdash_{j(\dB_{\kappa})/G}\check{\kappa}\in j(\tau)/G$ and thus $B\in I^+$, finishing the proof.
	\end{proof}
	
	The iterability of $(\alpha,\delta)$-semiproperness was shown by Shelah (see \cite[Chapter XII, Theorem 1.9]{ShelahProperImproper}). We will not give the definition of $\RCS$-iterations here and refer to the following two blackbox-theorems.
	
	The original definition of $\RCS$-iterations appears in \cite[Chapter X]{ShelahProperImproper}. Alternative constructions using boolean algebras can be found in unpublished notes by Donder-Fuchs \cite{FuchsDonderRCS} and Viale-Audrito-Steila \cite{VialeBooleanAlgebraicApproach}. A detailed proof of the iterability of $(\alpha,\omega_1)$-semiproperness using boolean algebras can be found in unpublished notes of the first author\footnote{https://hannesjakob.github.io/U01RCSIter.pdf}.
	
	\begin{mysen}\label{Theorem: Copy of Defining RCS Iteration}
		Let $\gamma$ be an ordinal and let $F$ be a function on $\gamma$. Then there is an $\RCS$-iteration $(\dB_i)_{i<\gamma+1}$ such that, for each $i<\gamma$, $\dB_{i+1}=\dB_{i}*F(i)$, provided $F(i)$ is a $\dB_{i}$-name for a boolean algebra.
	\end{mysen}
	
	\begin{mysen}\label{Theorem: Copy of RCS Iteration}
		Let $\alpha$ be a countable ordinal and $\overline{\dB}=(\dB_i)_{i<\gamma}$ an $\RCS$-iteration. Suppose the following:
		\begin{enumerate}
			\item For every $i<\gamma$,
			$$\Vdash_{\dB_i}\dB_{i+1}/\dB_i\text{ is $(\alpha,\omega_1)$-semiproper}$$
			\item For every even ordinal $i<\gamma$, $\Vdash_{\dB_{i+1}}|2^{\dB_i}|\leq\omega_1$.
		\end{enumerate}
		Then for every $i<j<\gamma$,
		$$\Vdash_{\dB_i}\dB_j/\dB_i\text{ is $(\alpha,\omega_1)$-semiproper}$$
	\end{mysen}
	
	To finish this section, we show that, by adding a proper collapse, it is possible increase the ``degree" of semiproperness:
	
	\begin{mylem}\label{Lemma: Two-Step Iteration}
		Let $\delta$ be a cardinal and $\alpha$ be a countable ordinal. Suppose that $\dP$ is $(\alpha,\delta)$-semiproper, forces $|\delta|=\omega_1$ and $\dot{\dQ}$ is a $\dP$-name for an $(\alpha,\omega_1)$-semiproper forcing. Then $\dP*\dot{\dQ}$ is $(\alpha,\delta)$-semiproper.
	\end{mylem}
	
	\begin{proof}
		Let $(M_i)_{i<\alpha}$ be an $\alpha$-tower of countable elementary submodels of $H(\Theta)$ where $\Theta$ is a sufficiently large regular cardinal such that $\dP*\dot{\dQ}, \delta, \alpha\in M_0$ and let $(p,\dot{q})\in M_0\cap\dP*\dot{\dQ}$. Let $p_0\leq p$ be $(M_i,\dP)$-$\delta$-semigeneric for every $i<\alpha$. Let $G$ be $\dP$-generic containing $p_0$. Then $(M_i[G])_{i<\alpha}$ is an $\alpha$-tower of countable elementary submodel of $H(\Theta)[G] = (H(\Theta))^{V[G]}$ and thus there is $\dot{q}_0^G\leq\dot{q}^G$ which is $M_i[G]$-semigeneric for every $i<\alpha$. Back in $V$, let $\dot{q}_0$ be forced to be an extension of $\dot{q}$ which is $M_i[\dot{G}_{\dP}]$-$\omega_1$-semigeneric for every $i<\alpha$. We claim that $(p_0,\dot{q}_0)$ is $M_i$-$\delta$-semigeneric for every $i<\alpha$. Let $G*H$ be $\dP*\dot{\dQ}$-generic containing $(p_0,\dot{q}_0)$. Then for every $i<\alpha$, $M_i[G*H]=M_i[G][H]$. Furthermore, $M_i[G]\cap\delta=M_i\cap\delta$. By our assumption on the forcing $\dP$, $M_i[G]$ contains a bijection $F$ between $\omega_1$ and $\delta$ and $\dot{q}_0$ is forced to be $M_i[G]$-$\omega_1$-semigeneric. Thus, $M_i[G][H]\cap\omega_1=M_i[G]\cap\omega_1$. However, since $F\in M_i[G]$ is a bijection between $\omega_1$ and $\delta$, this implies
		$$M_i[G][H]\cap\delta=M_i[G]\cap\delta=M_i\cap\delta$$ for any $i<\alpha$, as desired.
	\end{proof}
	
\section{The $\alpha$-$(\dagger)$ principle}\label{section: dagger}

In this section, we generalize the classical $(\dagger)$ principle (\cite{ForemanMagidorShelahMM}) which asserts that a forcing is stationary set preserving iff it is semiproper to the context where the relevant concepts are relativised to countable towers. These results will be applied in the proof of our main theorem to show the forcing iterands are $(\alpha, \omega_1)$-semiproper for all $\alpha<\omega_1$.

	\begin{mydef}[Definition 5.2, \cite{AbrahamProperChapter}]
		Let $A$ be an uncountable set and $\alpha$ a countable ordinal.
		\begin{enumerate}
			\item $P_{\omega_1}^{\alpha}(A)$ is the set of all $\subseteq$-increasing and continuous sequences $(a_i)_{i<\alpha}$ of countable subsets of $A$;
			\item An $\alpha$-function is a function $F\colon(\bigcup_{\beta<\alpha}P_{\omega_1}^{\beta}(A))\times [A]^{<\omega}\to P_{\omega_1}(A)$. A sequence $(a_i)_{i<\alpha}\in P_{\omega_1}^{\alpha}(A)$ is \emph{closed under $F$} if for every $\beta<\alpha$ which is a successor ordinal or $0$, for every $x\in[a_{\beta}]^{<\omega}$, $F((a_i)_{i<\beta},x)\subseteq a_{\beta}$.
			\item For an $\alpha$-function $F$, let $G(F)$ be the collection of all $\alpha$-sequences which are closed under $F$. Then $\{G(F)\;|\;F\text{ is an $\alpha$-function}\}$ generates a countably closed filter on $P_{\omega_1}^{\alpha}(A)$, which we denote by $\mathcal{D}_{\omega_1}^{\alpha}(A)$.
			\item We say that $S\subseteq P_{\omega_1}^{\alpha}(A)$ is \emph{stationary} if its complement is not in $\mathcal{D}_{\omega_1}^{\alpha}(A)$.
		\end{enumerate}
	\end{mydef}
	
	So, $S\subseteq P_{\omega_1}^{\alpha}(A)$ is stationary if and only if for every $\alpha$-function $F$, there is some $(a_i)_{i\in\alpha}\in S$ which is closed under $F$. To simplify the notion of $\alpha$-stationarity, we have the following proposition, which is an analogue of the folklore result that $S\subseteq[A]^{<\kappa}$ is stationary if and only if for every large enough $\Theta$ and every $x\in H(\Theta)$, there is $M\prec H(\Theta)$ of size ${<}\,\kappa$ such that $x\in M$ and $M\cap A\in S$.
	
	\begin{mypro}\label{Proposition: Equivalence of Stationarity}
		Let $A$ be an uncountable set and $\alpha$ an ordinal. Let $S\subseteq P_{\omega_1}^{\alpha}(A)$. The following are equivalent:
		\begin{enumerate}
			\item $S$ is stationary in $P_{\omega_1}^{\alpha}(A)$;
			\item For every sufficiently large $\Theta$ and every $x\in H(\Theta)$, there is an $\alpha$-tower $(M_i)_{i<\alpha}$ of countable elementary submodels of $H(\Theta)$ such that $x,A\in M_0$ and $(M_i\cap A)_{i<\alpha}\in S$.
		\end{enumerate}
	\end{mypro}
	
	\begin{proof}
		Suppose first that $S$ is stationary in $P_{\omega_1}^{\alpha}(A)$. Let $\Theta > 2^{|A|}$ be a regular cardinal. Let $x\in H(\Theta)$. Let $(G_i)_{i<\alpha}$ and $(\Theta_i)_{i<\alpha}$ be defined as follows: $\Theta_0=\Theta$ and $G_0$ is a well-ordering for $H(\Theta)$. For every successor $i<\alpha$, $\Theta_i$ is large enough to contain $(G_j)_{j<i}$ and $(\Theta_j)_{j<i}$ and $G_i$ is a well-ordering for $H(\Theta_i)$.
		
		Now let $F$ be defined as follows: For every successor $i<\alpha$, $(a_j)_{j<i}\in P_{\omega_1}^i(A)$ and $y\in[A]^{<\omega}$, we let
		$$F((a_j)_{j<i},y):=\Hull^{(H(\Theta_i),\in,G_i)}(\{(\Theta_j)_{j<i},(G_j)_{j < i},(a_i)_{j<i}, x, A\}\cup y)\cap A$$
		Since $S$ is stationary, we can fix $(a_i)_{i<\alpha}\in S$ which is closed under $F$. For every successor $i<\alpha$, let $N_i=\Hull^{(H(\Theta_i),\in,G_i)}(\{(\Theta_j)_{j<i},(G_j)_{j < i},(a_i)_{j<i}, x, A\}\cup a_i)$ and $M_i= N_i \cap H(\Theta)$. For limit $i<\alpha$, we take the union. We verify that $(M_i)_{i<\alpha}$ is as required.
		
		First note that, for every successor $i<\alpha$, $M_i\cap A=a_i$ by the definition of $F$. Since $(M_i)_{i<\alpha}$ and $(a_i)_{i<\alpha}$ are continuous, this persists in limits. Also, clearly $x,A\in M_0$, so we only have to verify that $(M_i)_{i<\alpha}$ is an $\alpha$-tower. 
		
	\begin{enumerate}
	\item Clearly $M_0\prec H(\Theta)$ by definition and when $i>0$, $M_i\prec H(\Theta)$ since $M_i\subseteq N_i\prec H(\Theta_i)$ and $(H(\Theta), G_0)\in H(\Theta_i)$. 
	\item For every successor $i<\alpha$, $(N_j)_{j<i}\in N_i$ by the definition of $N_i$. Since $(M_j)_{j<i}= (N_j \cap H(\Theta))_{j<i} \in N_i\cap H(\Theta)=M_i$. In particular, for any $j_0<j_1$, $M_{j_0}\subseteq M_{j_1} \prec H(\Theta)$.
	\end{enumerate}

		Now assume that (2) holds. Let $F$ be an $\alpha$-function and let $\Theta$ be sufficiently large. We may thus find an $\alpha$-tower $(M_i)_{i<\alpha}$ such that $F,A\in M_0$ and $(M_i\cap A)_{i<\alpha}\in S$. We verify that $(M_i\cap A)_{i<\alpha}$ is closed under $F$: For every successor ordinal $i<\alpha$, $(M_j\cap A)_{j<i}\in M_i$ by the definition of an $\alpha$-tower. Since $M_i\prec (H(\Theta),\in)$ and contains $F$, for every $x\in[M_i\cap A]^{<\omega}$, $F((a_j)_{j<i},x)\subseteq M_i\cap A$. So $(M_i\cap A)_{i<\alpha}$ is closed under $F$.
	\end{proof}
We need the following fact about generalized normality. 
\begin{mylem}\label{lemma: normality}
Let $A$ be an uncountable set and $\alpha$ an ordinal. Let $S\subseteq P_{\omega_1}^{\alpha}(A)$ be stationary and let $g: S\to A$ be a choice function in the sense that $g((M_i)_{i<\alpha})\in M_0$. Then there exists a stationary $S'\subseteq S$ and $p\in A$ such that $g$ is constant on $S'$ with value $p$.
\end{mylem}	

\begin{proof}
Suppose for the sake of contradiction that for any $p\in A$, $C_p=\{\bar{M}\in S: g(\bar{M})\neq p\}\in \mathcal{D}^\alpha_{\omega_1}(A)$. In other words, there is some $\alpha$-function $F_p$ such that $C_p \supseteq G(F_p)$. We claim that $E=\{\bar{M}: \forall p\in M_0, \bar{M}\in G(F_p)\}\in \mathcal{D}^\alpha_{\omega_1}(A)$. Let $\Theta$ be a sufficiently large regular cardinal and consider the following collection $$B=\{(N_i\cap A)_{i<\alpha}: (N_i)_{i<\alpha} \text{ is an $\alpha$-tower of elementary submodels of }H(\Theta) $$$$\text{ with $N_0$ containing } (F_p)_{p\in A}, A, \alpha\}.$$
Clearly, $B\in \mathcal{D}^\alpha_{\omega_1}(A)$. We show that $B\subseteq E$. Given $(N_i\cap A)_{i<\alpha}\in B$ and $p\in N_0\cap A$, we need to show that $(N_i\cap A)_{i<\alpha}\in G(F_p)$. Given a successor ordinal $j<\alpha$, since $(N_i\cap A)_{i<j}, F_p, A\in N_j$, we know that for any $a\in [N_j\cap A]^{<\omega}$, $F_p((N_i\cap A)_{i<j}, a) \in N_j \cap [A]^\omega \subseteq N_j\cap A$. It follows that $(N_i\cap A)_{i<\alpha}\in G(F_p)$ and $B\subseteq E$. Since $S$ is stationary, we can find $\bar{M}\in S\cap E$. Let $g(\bar{M})=p\in M_0$. However, by the definition of $E$, $\bar{M}\in C_p$ which means $g(\bar{M})\neq p$, which is a contradiction.
\end{proof}

	We can now prove that, in many models, a forcing being $\alpha$-stationary set preserving is equivalent to it being $(\alpha,\omega_1)$-semiproper. To start, we first show that a forcing being $(\alpha,\delta)$-semiproper implies that it preserves stationary subsets of $P_{\omega_1}^{\alpha}(\delta)$.
	
	\begin{mylem}\label{Lemma: Alpha-SP is Alpha-SSP}
		Let $S\subseteq P_{\omega_1}^{\alpha}(\delta)$ be stationary and let $\dP$ be a poset which is $(\alpha,\delta)$-semiproper. Then $S$ remains stationary after forcing with $\dP$.
	\end{mylem}
	
	\begin{proof}
	 By Proposition \ref{Proposition: Equivalence of Stationarity}, it suffices to show that in $V[G]$ where $G\subseteq \dP$ is generic over $V$, for every sufficiently large $\Theta$ and every $x\in H^{V[G]}(\Theta)$, there is an $\alpha$-tower $(M_i)_{i<\alpha}$ of countable elementary substructures of $H^{V[G]}(\Theta)$ such that $x,\delta\in M_0$ and $(M_i\cap\delta)_{i<\alpha}\in S$. To this end, suppose for the sake of contradiction that $\Theta$ is sufficiently large and $x\in H^{V[G]}(\Theta)$ witnesses that $S$ is nonstationary. 
	 
	Work in $V$. Let $\tau$ be a name for $x$ and let $p\in \dP$ force that $\tau$ witnesses the nonstationarity of $S$. We may assume that $\Theta$ is so large that $\tau,\dP\in H^V(\Theta)$. Since $S$ is stationary, there is an $\alpha$-tower $(M_i)_{i<\alpha}$ of countable elementary substructures of $H^V(\Theta)$ such that $\tau,\delta,\dP,p\in M_0$ and $(M_i\cap\delta)_{i<\alpha}\in S$.
		
		Since $\dP$ is $(\alpha,\delta)$-semiproper, we can find a condition $q\leq p$ which is $((M_i)_{i<\alpha},\dP)$-$\delta$-semigeneric. Let $H$ be any $\dP$-generic filter containing $q$. In $V[H]$, consider the sequence $(M_i[H])_{i<\alpha}$. It is easy to see that $(M_i[H])_{i<\alpha}$ is an $\alpha$-tower of countable elementary substructures of $H^{V[G]}(\Theta)$ and $\tau^H\in M_0[H]$. Furthermore, for every $i<\alpha$, $M_i[H]\cap\delta=M_i\cap\delta$ and so $(M_i[H]\cap\delta)_{i<\alpha}\in S$. But then $\tau^H$ does not witness that $S$ is nonstationary in $V[H]$, contradicting our assumptions.
	\end{proof}
	
	Now we can prove the other direction. Note that, clearly, if $\dP$ preserves stationary subsets of $P_{\omega_1}^{\alpha}(\delta)$, then $\dP$ preserves stationary subsets of $P_{\omega_1}^{\alpha}(A)$ for every set $A$ with $|A|=\delta$.
	
	\begin{mylem}\label{Lemma: Dagger for alpha-stationarity}
		Suppose $\kappa$ is supercompact and $\alpha$ is a countable ordinal. Let $\dP$ be a poset such that
		\begin{enumerate}
			\item $\dP$ has size $\kappa$, is $\kappa$-cc and forces $\check{\kappa}=\omega_2$;
			\item Whenever $G$ is $\dP$-generic and $\delta\geq\kappa$, there is a $\delta$-supercompact elementary embedding $j\colon V\to M$ such that $M[G] \models$ `` $j(\dP)/G$ is $(\alpha,\delta)$-semiproper".
		\end{enumerate}
		
		Let $G$ be $\dP$-generic. In $V[G]$, let $\dQ$ be a poset which preserves stationary subsets of $P_{\omega_1}^{\alpha}(\omega_1)$. Then $\dQ$ is $(\alpha,\omega_1)$-semiproper.
	\end{mylem}
	
	\begin{proof}
		
		Seeking a contradiction, suppose that $G$ is a $\dP$-generic filter and $\dQ\in V[G]$ is a counterexample. Thus, $\dQ$ is a poset which preserves stationary subsets of $P_{\omega_1}^{\alpha}(\omega_1)$ but $\dQ$ is not $(\alpha,\omega_1)$-semiproper. Hence, for every sufficiently large $\Theta$ and every $x\in H(\Theta)$, there is an $\alpha$-tower $\overline{M}$ of countable elementary substructures and $p\in M_0\cap\dQ$ such that there does not exist any $(\overline{M},\dQ)$-semigeneric condition $q\leq p$. Thus, letting $\lambda:=(2^{|\dP*\dQ|})^+$, the set of all $\alpha$-towers for which there is no $(\overline{M},\dQ)$-semigeneric condition below some $p\in M_0\cap\dQ$ is stationary in $P_{\omega_1}^{\alpha}(H^{V[G]}(\lambda)))$ by Proposition \ref{Proposition: Equivalence of Stationarity}. By Lemma \ref{lemma: normality}, there is a stationary subset $S$ of $P_{\omega_1}^{\alpha}(H^{V[G]}(\lambda))$ and a single condition $p\in\dQ$ such that for every $\overline{M}\in S$, $p\in M_0$ and there is no $(\overline{M},\dQ)$-semigeneric condition $q\leq p$. By replacing $\dQ$ with the poset $\{r\in\dQ\;|\;r\leq q\}$ if necessary, we may assume for simplicity that $q$ is the weakest condition in $\dQ$.
		
		By the assumption of the lemma, there is a $\delta:=|H^{V[G]}(\lambda)|$-supercompact embedding $j\colon V\to M$ such that $j(\dP)/G$ is $(\alpha,\delta)$-semiproper in $M[G]$. Note that $S\in M[G]$. Let $H$ be any $j(\dP)$-generic filter extending $G$ obtained by forcing with $j(\dP)/G$. By Lemma \ref{Lemma: Alpha-SP is Alpha-SSP}, $S$ is stationary in $P_{\omega_1}^{\alpha}(H^{V[G]}(\lambda))$ in $M[H]$. Moreover, in $V[H]$, $j$ lifts to $j^+\colon V[G]\to M[H]$, since $j''G =G\subseteq H$.
		
		In $M[H]$, $j^+(\dQ)$ preserves stationary subsets of $P_{\omega_1}^{\alpha}(\omega_1)$ by elementarity. Again by elementarity, in $M[H]$, $j(\kappa)=\omega_2$ and so $|\delta|=|H^{V[G]}(\lambda)|=\omega_1$. Ergo, $S$ is morally a stationary subset of $P_{\omega_1}^{\alpha}(\omega_1)$ and so its stationarity is preserved by forcing with $j^+(\dQ)$. Let $K$ be any $j^+(\dQ)$-generic filter over $M[H]$ and work in $M[H][K]$.
		
		In $M[H][K]$, consider the function $F$ mapping a $\dQ$-name $\tau\in H^{V[G]}(\lambda)$ to $j^+(\tau)^K$ (recall that $j'' H(\lambda)\subseteq M$). Since $S\subseteq P_{\omega_1}^{\alpha}(H^{V[G]}(\lambda))$ is stationary, by Proposition \ref{Proposition: Equivalence of Stationarity}, there is an $\alpha$-tower $(N_i)_{i<\alpha}$ of countable elementary substructures of some sufficiently large $H(\Theta)$ such that $F\in N_0$ and $(N_i\cap H^{V[G]}(\lambda))_{i<\alpha}\in S$. For $i<\alpha$, let $M_i:=N_i\cap H^{V[G]}(\lambda)$.
		
		Hence, letting $\dot{F}$ be a name for $F$, there is a condition $q\in j^+(\dQ)$ which forces that there exists an $\alpha$-tower $(N_i)_{i<\alpha}$ of countable elementary substructures of $H^{M[H][\dot{G}_{j^+(\dQ)}]}(\Theta)$ such that $\dot{F}\in N_0$ and $(N_i\cap H^{V[G]}(\lambda))_{i<\alpha}=(M_i)_{i<\alpha}$. In particular, whenever $\tau\in M_i$ is a $\dQ$-name for a countable ordinal, $q$ forces that $\dot{F}(\tau)=j^+(\tau)^{\dot{G}_{j^+(\dQ)}}\in M_i$, since $M_i\cap\omega_1=N_i\cap\omega_1$, where $N_i$ is closed under $\dot{F}$. In particular, whenever $\tau\in j^+[M_i]$ is a $j^+(\dQ)$-name for a countable ordinal, $q$ forces that $\tau^{\dot{G}_{j^+(\dQ)}}\in j^+[M_i]$, since $j^+$ does not move countable ordinals. Ergo, $q$ is a $(j^+((M_i)_{i<\alpha}), j^+(\dQ))$-semigeneric condition.
		
		By elementarity of $j^+$, there is a $((M_i)_{i<\alpha},\dQ)$-semigeneric condition. However, this contradicts our assumptions. Thus, the lemma is proved.
	\end{proof}

\begin{mybem}
A straightforward adaption of the arguments in \cite[Ch. XIII, \S 1]{ShelahProperImproper} gives the equivalence between the conclusion of Lemma \ref{Lemma: Dagger for alpha-stationarity} and the $\alpha$-tower version of the semi-stationary reflection ($\alpha$-$\mathrm{SSR}$). Fix a countable ordinal $\alpha$. Recall for $x,y$ countable subsets of ordinals, we let $x\sqsubseteq y$ to denote that $x\subseteq y$ and $x\cap \omega_1 = y \cap \omega_1$. For an uncountable set $A$ containing $\omega_1$, we say $S\subseteq P^\alpha_{\omega_1}(A)$ is \emph{semi-stationary} if $\{(b_i)_{i<\alpha} \in P^\alpha_{\omega_1}(A): \exists (a_i)_{i<\alpha}\in S, \forall i\in \alpha, a_i\sqsubseteq b_i\}$ is a stationary subset of $P^\alpha_{\omega_1}(A)$. 
$\alpha$-$\mathrm{SSR}$ asserts that for large enough regular $\Theta$ and for any semi-stationary $S\subseteq P^{\alpha}_{\omega_1}(H(\Theta))$, there exists a set $W\subseteq H(\Theta)$ of size $\aleph_1$ containing $\omega_1$ such that $S\cap P^\alpha_{\omega_1}(W)$ is a semi-stationary subset of $P^\alpha_{\omega_1}(W)$.
\end{mybem}

	\section{An equiconsistency result concerning local ideal games}\label{Section: Weak Ideal Game}
	
	In this section, we investigate the (local) game $\Game^{\alpha}(I, \mathcal{X})$ in depth.	By restricting the set of functions which can be played, we obtain these ideals from much lower assumptions. Let $I_{\text{bdd}}^{\cof(\omega_1)}$ be the ideal of all sets $A\subseteq\omega_2$ such that $A\cap\cof(\omega_1)$ is bounded in $\omega_2$. Our main result for this section is as follows: 

	\begin{mysen}\label{Theorem: Equiconsistency for Game}
		The following are equiconsistent:
		\begin{enumerate}
			\item There exists a weakly compact cardinal.
			\item For every $\omega_2$-sized collection $\mathcal{X}$ of regressive functions on $\omega_2$ and every ordinal $\alpha<\omega_1$, player II has a winning strategy in $\Game^{\alpha}(I_{\text{bdd}}^{\cof(\omega_1)},\mathcal{X})$.
		\end{enumerate}
	\end{mysen}
	
	We begin with the easier half of Theorem \ref{Theorem: Equiconsistency for Game} which is the upper bound. To this end, we show:
	
	\begin{mysen}\label{Theorem: Ideal Game Upper Bound}
		Suppose $\kappa$ is a weakly compact cardinal. Let $G$ be $\Coll(\omega_1,<\kappa)$-generic.
		
		In $V[G]$, whenever $\mathcal{X}$ is an $\omega_2$-sized collection of regressive functions on $\omega_2$ and $\alpha$ is a countable ordinal, player II wins $\Game^{\alpha}(I_{\text{bdd}}^{\cof(\omega_1)},\mathcal{X})$.
	\end{mysen}
	
	\begin{proof}
		In $V$, let $\dot{\mathcal{X}}$ be a $\kappa$-sized collection of $\Coll(\omega_1,<\kappa)$-names such that $\mathcal{X}=\{\tau^G\;|\;\tau\in\dot{\mathcal{X}}\}$. Also in $V$, suppose the following:
		\begin{enumerate}
			\item $X\prec H(\Theta)$ has size $\kappa$ containing $\kappa, \dot{\mathcal{X}}$ and ${}^{<\kappa} X\subseteq X$, $M$ is the transitive collapse of $X$;
			\item $N$ is a transitive set and $N^{<\kappa}\subseteq N$;
			\item $j\colon M\to N$ is an elementary embedding with critical point $\kappa$ and $j,M\in N$.
		\end{enumerate}
		
		It is well-known that such objects exist because of the weak compactness of $\kappa$ (see  \cite[Theorem 16.1]{CummingsHandbook} and \cite{Hauser}). From the properties of $M$ and $N$ it follows that $\Coll(\omega_1,<j(\kappa))=(\Coll(\omega_1,<j(\kappa)))^N$. Let $H$ be $\Coll(\omega_1,<j(\kappa))/G$-generic over $V[G]$. It follows that, in $V[H]$, $j$ lifts to $j^+\colon M[G]\to N[H]$. Thus, in $V[H]$, we can define a filter $U$ over $\kappa$ by letting $A\in U$ if and only if $A\in M[G]$ and $\kappa\in j^+(A)$. Then for any $A\subseteq\kappa$ with $A\in M[G]$, either $A\in U$ or $\kappa\smallsetminus A\in U$.
		
		Now let $f\in\mathcal{X}$. Then $f\in M[G]$ and so $j^+(f)\in N[H]$ is a regressive function on $j(\kappa)$. Let $\zeta:=j^+(f)(\kappa)$. Then, by the definition of $U$, $\zeta$ is the unique ordinal in $\kappa$ such that $\{\alpha<\kappa\;|\;f(\alpha)=\zeta\}\in U$.
		
		Fix a $\Coll(\omega_1,<j(\kappa))/G$-name $\dot{U}$ for $U$. Let player II play as follows: In any run $(f_{\beta},\xi_{\beta})_{\beta<\alpha}$ of $\Game^{\alpha}(I_{\text{bdd}}^{\cof(\omega_1)},\mathcal{X})$, player II constructs, on the side, a descending sequence $(p_{\beta})_{\beta<\alpha}$ of conditions in $\Coll(\omega_1,<j(\kappa))/G$ such that, for every $\beta<\alpha$, $p_{\beta}$ forces that $\zeta_{\beta}$ is the unique ordinal such that $\{\alpha<\kappa\;|\;f_{\beta}(\alpha)=\zeta_{\beta}\}\in\dot{U}$ and $\xi_{\beta}=\zeta_{\beta}+1$. Let us show that this is possible: After $(f_{\gamma},\xi_{\gamma})_{\gamma<\beta}$ has been played and $(p_{\gamma})_{\gamma<\beta}$ has been constructed on the side, let $p_{\beta}'$ be a lower bound of $(p_{\gamma})_{\gamma<\beta}$. Let player I play $f_{\beta}\in\mathcal{X}$. It follows that $p_{\beta}'$ forces that there is a unique ordinal $\dot{\zeta}_{\beta}$ such that $\{\alpha<\kappa\;|\;f_{\beta}(\alpha)=\zeta_{\beta}\}\in\dot{U}$. Let $p_{\beta}\leq p_{\beta}'$ decide the value of $\dot{\zeta}_{\beta}$ to be $\zeta_{\beta}$ and let player II play $\xi_{\beta}:=\zeta_{\beta}+1$.
		
		Assume that $(f_{\beta},\xi_{\beta})_{\beta<\alpha}$ is a run of the game where player II played according to this strategy and let $(p_{\beta})_{\beta<\alpha}$ be the sequence of conditions constructed by player II. We will show that player II wins. To see this, let $p_{\alpha}$ be a lower bound of $(p_{\beta})_{\beta<\alpha}$. Let $H$ be $\Coll(\omega_1,<j(\kappa))/G$-generic over $V[G]$ containing $p_{\alpha}$.
		
		Since $\Coll(\omega_1,<j(\kappa))/G$ is countably closed, the sequence $(f_{\beta},\xi_{\beta})_{\beta<\alpha}$ is an element of $V[G]$ and, in particular, an element of $M[G]$ as $V[G]\models {}^\omega M[G]\subseteq M[G]$. It follows that for any $\gamma<\kappa$,
		$$N[H]\models\cf(\kappa)=\omega_1\wedge\kappa>\gamma\wedge\forall\beta<\alpha(j^+(f_{\beta})(\kappa)=\zeta_{\beta})$$
		By the elementarity of $j^+$, the set of all $\eta\in\kappa$ such that $\cf(\eta)=\omega_1$ and for all $\beta<\alpha$, $f_{\beta}(\eta)=\zeta_{\beta}<\xi_{\beta}$, is unbounded in $\kappa$.
	\end{proof}
	
	Now we concern ourselves with the lower bound. We begin with a technical lemma which allows us to show that a winning strategy for player II in $\Game^{\alpha}(I)$ enables the construction of Chang-type models. For concretness, we focus on the case $\kappa=\omega_2$ and $I=I_{\text{bdd}}^{\cof(\omega_1)}$, although the methods are more general.
	
	\begin{mylem}\label{Lemma: Reduced Skolem Hull}
		Suppose that $\alpha<\omega_1$ is an ordinal such that player II has a winning strategy in $\Game^{\omega\alpha}(I_{\text{bdd}}^{\cof(\omega_1)},\mathcal{X})$. Then whenever $(M_i)_{i<\alpha}$ is an $\alpha$-tower of countable elementary substructures of $H(\Theta)$ for regular $\Theta\geq\omega_3$ with $\mathcal{X},\alpha\in M_0$, there is $X\subseteq E_{\omega_1}^{\omega_2}$ unbounded such that for every $\eta\in X$, $i<\alpha$ and $f\in M_i\cap\mathcal{X}$, $f(\eta)<\sup(M_i\cap\omega_2)$.
	\end{mylem}
	
	\begin{proof}
		Let $\sigma\in M_0$ be a winning strategy for player II in $\Game^{\omega\alpha}(I_{\text{bdd}}^{\cof(\omega_1)},\mathcal{X})$, which exists by the elementarity of $M_0$.
		
		Let $(f_{\beta},\xi_{\beta})_{\beta<\omega\alpha}$ be a run of the game in which player II plays according to $\sigma$ and for every ordinal $\gamma<\alpha$, $\{f_{\beta}\;|\;\beta<\omega(i+1)\}$ enumerates $\mathcal{X}\cap M_i$. It follows that for every $\beta<\omega(i+1)$, $(f_{\gamma})_{\gamma\leq\beta}\in M_i$ and so $\xi_{\beta}\in M_i$. In particular, $\sup_{\beta<\omega(i+1)}\xi_{\beta}\leq \sup(M_{\gamma}\cap\omega_2)$.
		
		Now let $X\subseteq E_{\omega_1}^{\omega_2}$ be unbounded and witness that player II wins the run of the game. It follows that for every $\eta\in X$, $i<\alpha$ and $f\in M_i\cap\mathcal{X}$, $f$ equals $f_{\beta}$ for some $\beta<\omega(i+1)$ and so $f(\eta)=f_{\beta}(\eta)<\xi_{\beta}<\sup(M_i\cap\omega_2)$.
	\end{proof}
	
	Using the following well-known observation, we can connect the previous lemma to variants of Chang's conjecture:
	
	\begin{mylem}\label{Lemma: Skolem Hull Presentation}
		Let $\Theta$ be a regular cardinal. Let $\mathcal{A}$ be an algebra on $H(\Theta)$ containing a well order $<_{\Theta}$ of $H(\Theta)$ and $M\prec \mathcal{A}$. Let $x\subseteq\delta$ be countable, where $\delta\in M$. Then
		$$\Hull^{\mathcal{A}}(M\cup x)=\{f(y)\;|\;f\colon \delta^{<\omega}\to V,f\in M,y\in[x]^{<\omega}\}$$
		Moreover, for any $A\in M$,
		$$\Hull^{\mathcal{A}}(M\cup x)\cap A=\{f(y)\;|\;f\colon\delta^{<\omega}\to A,f\in M, y\in[x]^{<\omega}\}$$
	\end{mylem}
	
	\begin{proof}
		Clearly, the set
		$$M^x:=\{f(y)\;|\;f\colon \delta^{<\omega}\to V, f\in M, y\in[x]^{<\omega}\}$$
		is contained in $\Hull^{\mathcal{A}}(M\cup x)$, so we only have to show that $M^x$ is an elementary substructure of $\mathcal{A}$ containing $M$ and $x$. The latter statement is obvious. To see that $M^x\prec\mathcal{A}$, suppose that $\mathcal{A}\models\exists x\phi[f_0(y_0),\dots,f_{n-1}(y_{n-1}),x]$. We may suppose that $y_0=y_1=\dots=y_{n-1}=:y$. Let $f$ be the function mapping a finite set $z$ of size $|y|$ to the $<_{\Theta}$-least element $x$ satisfying $\phi[f_0(z),\dots,f_{n-1}(z),x]$ if it exists, $0$ otherwise. Then $f\in M$ and therefore, $f(y)\in M^x$. By construction, $\mathcal{A}\models\phi[f_0(y_0),\dots,f_{n-1}(y_0),f(y)]$ and so we are done by Tarski's criterion.
		
		To see the ``moreover'' part, note that for any function $f$, we can define $f_A$ with $f_A(z)=f(z)$ if $f(z)\in A$ and $f_A(z)$ the $<_{\Theta}$-least element of $A$ otherwise. If $f\in M$, so is $f_A$ by elementarity.
	\end{proof}
	
	The previous lemma allows us to show that, in these specific cases, the Skolem hull of a sufficiently nice structure is definable without appealing to the overarching structure $H(\Theta)$ and can thus be reasoned about by elementary submodels of $H(\Theta)$.
	
	\begin{mylem}\label{Lemma: Alpha-Chang's Conjecture}
		Suppose that $\alpha<\omega_1$ is an ordinal such that player II has a winning strategy in $\Game^{\omega\alpha}(I_{\text{bdd}}^{\cof(\omega_1)})$. Then whenever $\mathcal{A}$ is an algebra on $H(\Theta)$ for regular $\Theta\geq\omega_3$ containing a well-oder $<_{\Theta}$ of $H(\Theta)$ and $(M_i)_{i<\alpha}$ is an $\alpha$-tower of countable elementary substructures of $H(\Theta)$ with $\alpha\in M_0$, there is $X\subseteq E_{\omega_1}^{\omega_2}$ unbounded such that for every $\eta\in X$ and $i<\alpha$, $\Hull^{\mathcal{A}}(M_i\cup\{\eta\})\cap\eta\subseteq\sup(M_i\cap\omega_2)$.
	\end{mylem}
	
	\begin{proof}
		By Lemma \ref{Lemma: Reduced Skolem Hull}, we can find a set $X\subseteq E_{\omega_1}^{\omega_2}$ unbounded such that for every $\eta\in X$, $i<\alpha$ and regressive function $f\in M_i$ on $\omega_2$, $f(\eta)<\sup(M_i\cap\omega_2)$. As in the proof of Lemma \ref{Lemma: Skolem Hull Presentation}, this implies that $\Hull^{\mathcal{A}}(M_i\cup\{\eta\})\cap\eta\subseteq\sup(M_i\cap\omega_2)$: Let $f\in M_i$ be any function on $\omega_2$. Let $f'$ be defined such that $f'(\eta)=f(\eta)$ provided $f'(\eta)<\eta$ and $f'(\eta)=0$ otherwise. Then $f'\in M_i$ and $f'$ is regressive. It follows that, whenever $f\in M_i$ is such that $f(\eta)<\eta$, $f(\eta)<\sup(M_i\cap\omega_2)$.
	\end{proof}
	
	Note that the other direction of Lemma \ref{Lemma: Alpha-Chang's Conjecture} holds as well, if we reduce the length of the game. This is due to the fact that, unlike in the classical game characterizing strong Chang's conjecture, we require that the functions are bounded by the exact ordinal played by player II and not by their supremum (see \cite[Chapter XII, Definition 2.1]{ShelahProperImproper}).
	
	\begin{mysen}\label{Theorem: Alpha-Chang's Conjecture}
		Suppose that $\alpha<\omega_1$ is an ordinal. Let $\Theta\geq\omega_3$ and let $\mathcal{A}$ be an algebra on $H(\Theta)$ contianing a well order $<_{\Theta}$ of $H(\Theta)$ such that whenever $(M_i)_{i<\alpha}$ is an $\alpha$-tower of countable elementary substructures of $\mathcal{A}$, there is $X\subseteq E_{\omega_1}^{\omega_2}$ unbounded such that for every $\eta\in X$ and $i<\alpha$,
		$$\Hull^{\mathcal{A}}(M_i\cup\{\eta\})\cap\eta\subseteq\sup(M_i\cap\omega_2)$$
		
		Then player II has a winning strategy in $\Game^{\alpha}(I_{\text{bdd}}^{\cof(\omega_1)})$.
	\end{mysen}
	
	\begin{proof}
		Let $(f_{\beta})_{\beta<\alpha}$ be a run of $\Game^{\alpha}(I_{\text{bdd}}^{\cof(\omega_1)})$. We will assume for simplicity that the trivial function is played for limit $\beta$. On the side, player II has constructed an $\alpha$-tower $(M_i)_{i<\alpha}$ of countable elementary substructures of $\mathcal{A}$ such that $f_i\in M_i$ for successor $i$ and played $\xi_i:=\sup(M_i\cap\omega_2)$. Let us show that this constitutes a winning strategy.
		
		By assumption, there is $X\subseteq E_{\omega_1}^{\omega_2}$ unbounded such that for every $\eta\in X$ and $i<\alpha$,
		$$\Hull^{\mathcal{A}}(M_i\cup\{\eta\})\cap\eta\subseteq\sup(M_i\cap\omega_2)$$
		In particular, for every $\eta\in X$, since $f_i(\eta)\in\Hull^{\mathcal{A}}(M_i\cup\{\eta\})$, $f_i(\eta)<\sup(M_i\cap\omega_2)=\xi_i$. Ergo, $X$ witnesses that II wins the run of $\Game^{\alpha}(I_{\text{bdd}}^{\cof(\omega_1)})$.
	\end{proof}
	
	In particular, whenever $\omega\alpha=\alpha$, the two conditions are equivalent.
	
	Lastly, we use Lemma \ref{Lemma: Reduced Skolem Hull} to show that, for a suitable collection $\mathcal{X}$, player II having a winning strategy in $\Game^{\omega+\omega}(I_{\text{bdd}}^{\cof(\omega_1)},\mathcal{X})$ implies that any two stationary subsets of $E_{\omega}^{\omega_2}$ reflect simultaneously. By work of Magidor (see \cite{MagidorReflectionStat}), this implies that $\omega_2$ is weakly compact in the constructible universe.
	
	\begin{mylem}
		Suppose that for every $\omega_2$-sized collection $\mathcal{X}$ of regressive functions on $\omega_2$, player II wins $\Game^{\omega+\omega}(I_{\text{bdd}}^{\cof(\omega_1)},\mathcal{X})$. Then for every pair $S_0,S_1$ of stationary subsets of $E^{\omega_2}_\omega$, there exists an ordinal $\gamma\in E_{\omega_1}^{\omega_2}$ such that $S_0\cap\gamma$ and $S_1\cap\gamma$ are stationary in $\gamma$.
	\end{mylem}
	
	\begin{proof}
		Suppose otherwise. It follows that, for every ordinal $\gamma\in E_{\omega_1}^{\omega_2}$, we can fix a set $c_{\gamma}\subseteq\gamma$ which is club and disjoint from either $S_0$ or $S_1$.
		
		Let $\Theta$ be a sufficiently large regular cardinal and let $\mathcal{X}$ be the collection of all regressive functions on $\omega_2$ which are in $\Hull^{(H(\Theta),\in)}(\omega_2\cup\{S_0,S_1,(c_{\gamma})_{\gamma\in E_{\omega_1}^{\omega_2}}\})$. Clearly, $|\mathcal{X}|=\omega_2$.
		
		Let $M_0\prec H(\Theta)$ be such that $(c_{\gamma})_{\gamma\in E_{\omega_1}^{\omega_2}}\in M_0$ and $M_0\cap\omega_2\in S_0$. Let $M_1\prec H(\Theta)$ be such that $(c_{\gamma})_{\gamma\in E_{\omega_1}^{\omega_2}},M_0\in M_1$ and $M_1\cap\omega_2\in S_1$. By Lemma \ref{Lemma: Reduced Skolem Hull}, there is an unbounded set of ordinals $\eta\in E_{\omega_1}^{\omega_2}$ such that for $i=0,1$ and any $f\in M_i\cap\mathcal{X}$, $f(\eta)<\sup(M_i\cap\omega_2)$. For $i=0,1$, let $\delta_i:=\sup(M_i\cap\omega_2)$. Choose $\eta$ as above with $\eta>\delta_0,\delta_1$.
		
		\begin{myclaim}
			For $i=0,1$, $c_{\eta}\cap\delta_i$ is unbounded in $\delta_i$.
		\end{myclaim}
		
		\begin{proof}
			We prove the claim simultaneously for both $i$. So let $\beta<\delta_i$ and assume that $\beta\in M_i\cap\omega_2$. Let $f$ be the function mapping an ordinal $\xi$ to the smallest element of $c_{\xi}$ above $\beta$ if $\cf(\xi)=\omega_1$ and such an element exists or $0$ otherwise. It follows that $f\in\mathcal{X}\cap M_i$ and so $f(\eta)<\sup(M_i\cap\omega_2)$. In particular, there is some $\beta'<\sup(M_i\cap\omega_2)$ with $\beta'\in c_{\eta}-(\beta+1)$.
		\end{proof}
		
		Since $c_{\eta}$ is club for every $\eta\in E_{\omega_1}^{\omega_2}$, $\delta_0,\delta_1\in c_{\eta}$. However, by assumption, $\delta_0\in S_0$ and $\delta_1\in S_1$, so $c_{\eta}$ is disjoint from neither $S_0$ nor $S_1$, a contradiction.
	\end{proof}
	
	\begin{mycol}
		Suppose that for every $\omega_2$-sized collection $\mathcal{X}$ of regressive functions on $\omega_2$, player II wins $\Game^{\omega+\omega}(I_{\text{bdd}}^{\cof(\omega_1)},\mathcal{X})$. Then $\omega_2$ is weakly compact in $L$.
	\end{mycol}
	
	\section{Adding Clubs into Stationary Sets}
	\label{section: adding clubs}
	The following poset was defined by Foreman-Magidor-Shelah in \cite[Theorem 9]{ForemanMagidorShelahMM} in order to show that Martin's Maximum implies that for every regular cardinal $\kappa\geq\omega_2$, every stationary subset of $E_{\omega}^{\kappa}$ contains a closed set of ordertype $\omega_1$.
	
	\begin{mydef}
		Let $S\subseteq E_{\omega}^{\omega_2}$ be stationary. The poset $\dP(S)$ consists of continuous and increasing functions $p\colon\alpha+1\to S$, where $\alpha<\omega_1$ is a countable ordinal, ordered by extension.
	\end{mydef}
	
	\begin{mysen}[Foreman-Magidor-Shelah]
		Let $S\subseteq E_{\omega}^{\omega_2}$ be stationary. The poset $\dP(S)$ preserves stationary subsets of $\omega_1$, is ${<}\,\omega_1$-distributive and adds a continuous and increasing function $f\colon\omega_1\to S$.
	\end{mysen}
	
	As is often the case for non-proper forcings, the poset $\dP(S)$ is not necessarily semiproper, although it always preserves stationary subsets of $\omega_1$. In the following, we show that $\dP(S)$ satisfies the corresponding notion of being ``$\alpha$-stationary set preserving'' and thus, by Lemma \ref{Lemma: Dagger for alpha-stationarity}, is $(\alpha,\omega_1)$-semiproper in many models.

	\begin{mysen}\label{Theorem: Alpha-Stationary Preservation}
		Let $S\subseteq E_{\omega}^{\omega_2}$ be stationary. Then for every $\alpha\in\omega_1$, $\dP(S)$ preserves stationary subsets of $P_{\omega_1}^{\alpha}(\omega_1)$.
	\end{mysen}
	
	The first step in Theorem \ref{Theorem: Alpha-Stationary Preservation} is an analogue of Friedman's classical result that, for every regular cardinal $\kappa\geq\omega_1$, every stationary subset of $E_{\omega}^{\kappa}$ contains arbitrarily long countable closed subsets (\cite{FriedmanClosed}).
	
	\begin{mypro}\label{Proposition: Alpha-Projective Stationarity}
		Let $S\subseteq E_{\omega}^{\omega_2}$ be stationary and let $\alpha$ be a countable ordinal. Then for every sufficiently large $\Theta$ and every $x\in H(\Theta)$, there is a continuous sequence $(M_i)_{i<\alpha}$ of $\omega_1$-sized elementary submodels of $H(\Theta)$ such that $x\in M_0$ and for each ordinal $i<\alpha$, $M_i\cap\omega_2\in S$ and if $i$ is a successor, $(M_j)_{j<i}\in M_i$.
	\end{mypro}
	
	\begin{proof}
		We prove the statement by induction on $\alpha\in \omega_1$. Suppose it holds for all $\beta<\alpha$.
		\begin{itemize}
			\item Case 1: $\alpha$ is a limit. Let $(\alpha_n)_{n\in\omega}$ be a sequence of successor ordinals converging to $\alpha$, where $\alpha_0=0$. For each $n\in\omega$, find a sequence $(N_i^n)_{i<\alpha_n}$ witnessing the inductive hypothesis such that for every $n\in\omega$ and $k<n$, $(N_i^k)_{i<\alpha_k}\in M_0^n$. Let $(M_i)_{i<\alpha}$ be defined such that, for $i\in[\alpha_n,\alpha_{n+1})$, $M_i=N_i^{n+1}$. We show that $(M_i)_{i<\alpha}$ is as required: The only potentially problematic part is showing that $(M_j)_{j<i}\in M_i$. So let $i<\alpha$ be a successor ordinal, $i\in[\alpha_n,\alpha_{n+1})$. Then $M_i$ contains $(N_j^{n+1})_{j<i}$, and $(N_j^k)_{j<\alpha_k}$ for every $k<n$. But from these sequences we can easily construct $(M_j)_{j<i}$.
			\item Case 2: $\alpha$ is the successor of a successor ordinal. Let $\alpha=\beta+1$, where $\beta$ is a successor ordinal. Find a sequence $(M_i)_{i<\beta}$ witnessing the inductive hypothesis. Let $M_{\beta}$ be any $\omega_1$-sized elementary submodel of $H(\Theta)$ such that $M_{\beta}\cap\omega_2\in S$ and $(M_i)_{i<\beta}\in M_{\beta}$. Clearly, $(M_i)_{i<\alpha}$ is as required.
			\item Case 3: $\alpha$ is the successor of a limit ordinal. Let $\alpha=\beta+1$ for a limit ordinal $\beta$. Let $(\beta_n)_{n\in\omega}$ be an increasing sequence of successor ordinals converging to $\beta$ with $\beta_0=0$. Using the inductive hypothesis, define a sequence $(\overline{M}_{\xi})_{\xi<\omega_2}$ such that for each $\xi<\omega_2$:
			\begin{enumerate}
				\item $\overline{M}_{\xi}=(M_i^{\xi})_{i<\beta}$;
				\item for every $i<\beta$, $|M_i^{\xi}|=\omega_1$ and $M_i^{\xi}\cap\omega_2\in S$;
				\item for every successor $i<\beta$, $(M_j^{\xi})_{j<i},(\overline{M}_{\eta})_{\eta<\xi}\in M_i^{\xi}$
			\end{enumerate}
			
			Since $S$ is stationary, we can find $M_{\beta}'\prec H(\Theta)$ with size $\omega_1$ such that $(\overline{M}_{\xi})_{\xi<\omega_2}\in M_{\beta}'$ and $M_{\beta}'\cap\omega_2\in S$. Fix an increasing sequence $(\xi_n)_{n\in\omega}$ of successor ordinals cofinal in $M_{\beta}'\cap\omega_2$. Finally, define $(M_i)_{i\leq\beta}$ as follows: For every $i<\alpha$, if $i\in[\beta_n,\beta_{n+1})$, let $M_i:=M_i^{\xi_n}$. Since $(\xi_n)_{n\in\omega}$ is cofinal in $M_{\beta}'\cap\omega_2$, we can let $M_{\beta}:=\bigcup_{i\in\beta}M_i$ and find that $\sup(M_{\beta}\cap\omega_2)=\sup(M_{\beta}'\cap\omega_2)\in S$. Also, the sequence is continuous below $\alpha$ since we stitched it together at successor ordinals. Lastly, let $i<\alpha$ be a successor ordinal. Since $\beta$ is a limit, $i<\beta$. Let $n\in\omega$ be such that $i\in[\beta_n,\beta_{n+1})$. Then, by construction, $(M_j^{\xi})_{j<i}$, $(\overline{M}_{\eta})_{\eta<\xi}\in M_i^{\xi}=M_i$. But, from these parameters (and $(\xi_k)_{k<n}$ which is a member of $M_i^{\xi}$ due to its finiteness), we can construct $(M_j)_{j<i}$.
		\end{itemize}
	\end{proof}
	
	\begin{proof}[Proof of Theorem \ref{Theorem: Alpha-Stationary Preservation}]
	We may without loss of generality assume $\alpha$ is a successor ordinal.
		Let $T\subseteq P_{\omega_1}^{\alpha}(\omega_1)$ be stationary. We want to show that $T$ remains stationary after forcing with $\dP$. Fix a $\dP$-name for an $\alpha$ function $\dot{F}: \bigcup_{\beta<\alpha} P_{\omega_1}^{\beta}(\omega_1) \times [\omega_1]^{<\omega} \to \omega_1$. Fix some large enough regular cardinal $\Theta$. It suffices to show that there exists an $\alpha$-tower of countable elementary submodels $(N_i)_{i<\alpha}$ of $(H(\Theta), \in , <_{\theta},\dot{F})$, where $<_{\Theta}$ is a well order, such that 
		\begin{itemize}
		\item $(N_i\cap \omega_1)_{i<\alpha}\in T$, and 
		\item for all $i<\alpha$, $\sup N_i\cap \omega_2\in S$.
		\end{itemize}
To see why this is good, we can recursively construct a condition $p\in \dP(S)$ that is $(N_i, \dP(S))$-generic for all $i<\alpha$. Therefore, $p$ forces $(N_i\cap \omega_1)_{i<\alpha}=(N_i[G]\cap \omega_1)_{i<\alpha}$ is closed under $\dot{F}$, as desired.

By Proposition \ref{Proposition: Alpha-Projective Stationarity}, there is a continuous sequence $(M_i)_{i<\alpha}$ of $\omega_1$-sized elementary submodels of $\mathcal{H}=(H(\Theta),\in, <^*)$ such that $T,S, \dot{F}\in M_0$ and for each ordinal $i<\alpha$, $M_i\cap\omega_2\in S$ and if $i$ is a successor, $(M_j)_{j<i}\in M_i$. For each successor $i<\alpha$, let $x_i\subseteq M_i$ be the $<^*$-least that is countable and cofinal. 
Consider the following $\alpha$-function $G: \bigcup_{\beta<\alpha} P^\beta_{\omega_1}(\omega_1)\times [\omega_1]^{<\omega} \to \omega_1$: for $\beta$ either a successor ordinal or 0, $G((a_j)_{j<\beta}, x)= \Hull^{M_\beta}(\{(M_j)_{j<\beta}, (a_j)_{j<\beta}\}\cup x \cup x_\beta)\cap \omega_1$. Since $T$ is a stationary subset of $P^\alpha_{\omega_1}(\omega_1)$, we can find $(a_j)_{j<\alpha}\in T$ that is closed under $G$. For each $\beta<\alpha$ that is either a successor ordinal or $0$, define $N_\beta = \Hull^{M_\beta}(\{(M_j)_{j<\beta}, (a_j)_{j<\beta}\}\cup a_\beta \cup x_\beta)$. For $\beta$ limit, let us define $N_\beta=\bigcup_{i<\beta} N_i$.

	\begin{enumerate}
	\item for any $\beta<\alpha$, $\langle N_i: i\leq \beta\rangle\in N_{\beta+1}$. The reason is that $N_{\beta+1}$ contains $(M_j)_{j\leq \beta}$ and $(a_j)_{j\leq \beta}$ (note that $(x_j)_{j\leq \beta}$ is definable from $(M_j)_{j\leq \beta}$), which are exactly what is needed to define $\langle N_i: i\leq \beta\rangle$. 
	\item $(N_j\cap \omega_1)_{j<\alpha} = (a_j)_{j < \alpha}\in T$, 
	\item for any $j<\alpha$, $\sup N_j\cap \omega_2 = M_j\cap \omega_2\in S$.
	\end{enumerate}


	\end{proof}
	
	By Lemma \ref{Lemma: Dagger for alpha-stationarity}, the previous result implies that, in many models, the poset $\dP(S)$ is $(\alpha,\omega_1)$-semiproper for every $\alpha<\omega_1$. However, using a game very similar to $\Game^{\alpha}(I)$, we can obtain the same conclusion under weaker hypotheses. The crucial difference is that we only allow player I to play functions from $\omega_2$ to $\omega_1$ and thus also only allow player II to play ordinals in $\omega_1$.
	
	\begin{mydef}
		Let $I$ be an ideal on $\omega_2$ and let $\alpha$ be a countable ordinal. Let $\mathcal{X}$ be a set of functions from $\omega_2$ into $\omega_1$. We define the game $\Game_{\omega_1}^{\alpha}(I,\mathcal{X})$ as follows: The game lasts $\alpha$ many rounds. In round $\beta$, player I plays a function $f_{\beta}\in\mathcal{X}$. Player II responds with an ordinal $\xi_{\beta}<\omega_1$. At the end of the game, player II wins if and only if there is a set $X\in I^+$ such that for every $\eta\in X$ and $\beta<\alpha$, $f_{\beta}(\eta)\leq\xi_{\beta}$.
	\end{mydef}
	
	As before, we omit $\mathcal{X}$ if every function is allowed. The same way as in Section \ref{Section: Weak Ideal Game}, we can see the following:
	
	\begin{mylem}
		Suppose that $I$ is a ${<}\,\omega_2$-complete ideal on $\omega_2$ such that there is a set $\mathcal{B}\subseteq I^+$ which is dense with respect to $\subseteq$ and ${<}\,\omega_1$-closed. Then II has a winning strategy in $\Game_{\omega_1}^{\alpha}(I)$ for every $\alpha<\omega_1$.
	\end{mylem}
	
	In particular, whenever $\kappa$ is measurable and $G$ is $\Coll(\omega_1,<\kappa)$-generic, there is an ideal $I$ on $\omega_2$ such that player II has a winning strategy in $\Game_{\omega_1}^{\alpha}(I)$ for every $\alpha<\omega_1$ (\cite{GalvinJechMagidorIdealGame}).
	
	On the other hand, we also have the following, which is shown the same way as Lemma \ref{Lemma: Alpha-Chang's Conjecture}:
	
	\begin{mylem}\label{lemma: alphaSCC}
		Suppose that $\alpha<\omega_1$ is an ordinal such that player II has a winning strategy in $\Game_{\omega_1}^{\omega\alpha}(I_{\text{bdd}}^{\cof(\omega_1)})$. Then whenever $\Theta$ is a large enough regular cardinal and $\mathcal{A}= (H(\Theta), \in, \alpha, <_{\theta})$ where $<_\Theta$ is a well order on $H(\Theta)$ and $(M_i)_{i<\alpha}$ is an $\alpha$-tower of countable elementary substructures of $\mathcal{A}$, there is $X\subseteq\omega_2\cap\cof(\omega_1)$ unbounded such that for every $\eta\in X$ and $i<\alpha$, $\Hull^{\mathcal{A}}(M_i\cup\{\eta\})\cap\omega_1=M_i\cap\omega_1$.
	\end{mylem}
	
	\begin{mybem}
		The conclusion of Lemma \ref{lemma: alphaSCC} is a version of the strong Chang's conjecture \cite{ForemanMagidorShelahMM} concerning countable towers. We may use $\alpha$-cofinal-Strong Chang's Conjecture ($\alpha$-$\mathrm{SCC}^{\cof}$) to denote it and use $<\omega_1$-$\mathrm{SCC}^{\cof}$ to denote that for all $\alpha<\omega_1$, $\alpha$-$\mathrm{SCC}^{\cof}$. See \cite{cox2020adjoiningthingswantsurvey} for a comprehensive survey on related topics.
	\end{mybem}
	
	Using the previous lemma, we can show:
	
	\begin{mysen}\label{Theorem: Extend alpha-tower into stationary set}
		Suppose player II has a winning strategy in $\Game_{\omega_1}^{\omega\alpha}(I_{\text{bdd}}^{\cof(\omega_1)})$. Let $S\subseteq E_{\omega}^{\omega_2}$ be stationary. Let $\Theta$ be a sufficiently large regular cardinal and let $\mathcal{A}$ be an algebra on $H(\Theta)$ extending $(H(\Theta),\in,S, \alpha, <_\Theta)$ where $<_{\Theta}$ is a well order on $H(\Theta)$. Whenever $(M_i)_{i<\alpha}$ is an $\alpha$-tower of countable elementary submodels of $\mathcal{A}$, there is an $\alpha$-tower $(M_i')_{i<\alpha}$ of countable elementary submodels of $\mathcal{A}$ such that for every $i<\alpha$, $M_i\subseteq M_i'$, $M_i\cap\omega_1=M_i'\cap\omega_1$ and $\sup(M_i\cap\omega_2)\in S$.
	\end{mysen}
	
	\begin{proof}
		We prove the following claim by induction. For technical reasons, we focus on the cases where $\alpha$ is a successor.
		
		\begin{myclaim}
			Let $\alpha$ be a countable successor ordinal. Whenever $(M_i)_{i<\alpha}$ is an $\alpha$-tower of countable elementary submodels of $\mathcal{A}$, there is a sequence $(x_i)_{i<\alpha}$ of countable subsets of $\omega_2$ such that, for every $i<\alpha$,
			$$\Hull^{\mathcal{A}}(M_i\cup x_i)\cap\omega_1=M_i\cap\omega_1\text{ and }\sup\left(\Hull^{\mathcal{A}}(M_i\cup x_i)\cap\omega_2\right)\in S$$
			Furthermore, if $i$ is a successor, $(x_j)_{j<i}\in \Hull^{\mathcal{A}}(M_i\cup x_i)$ and if $i$ is a limit, $\Hull^{\mathcal{A}}(M_i\cup x_i)=\bigcup_{j<i}\Hull^{\mathcal{A}}(M_j\cup x_j)$.
		\end{myclaim}
		
		\begin{proof}
			First let $\alpha=1$, so there is just one model $M$. Let $N\prec H(\Theta)$ be a model with size $\omega_1$ such that $\omega_1\subseteq N$, $M\in N$ and $N\cap\omega_2\in S$. In particular, there is a countable increasing sequence $(\delta_n)_{n\in\omega}$ converging to $N\cap\omega_2$. Recursively apply Lemma \ref{lemma: alphaSCC} inside $N$ to find a sequence $(\eta_n)_{n\in\omega}$ such that for each $n\in\omega$, $\eta_n>\delta_n$ and $\Hull^{\mathcal{A}}(M_i\cup\{\eta_0,\dots,\eta_n\})\cap\omega_1=\Hull^{\mathcal{A}}(M_i\cup\{\eta_0,\dots,\eta_{n-1}\})$. Note that $N$ can define $\Hull^{\mathcal{A}}(M_i\cup\{\eta_0,\dots,\eta_{k-1}\})$ thanks to Lemma \ref{Lemma: Skolem Hull Presentation}.
			
			Let $M':=\bigcup_{n\in\omega}\Hull^{\mathcal{A}}(M\cup\{\eta_0,\dots,\eta_{n-1}\})$. It follows that $M'\subseteq N$ and thus, by construction, $\sup(M'\cap\omega_2)=N\cap\omega_2\in S$. Furthermore, by induction,
			$$M'\cap\omega_1=\bigcup_{n\in\omega}\Hull^{\mathcal{A}}(M\cup\{\eta_0,\dots,\eta_{n-1}\})\cap\omega_1=M\cap\omega_1$$
			
			So $x:=\{\eta_0,\eta_1,\dots\}$ is as required.
			
			Now suppose that $\alpha=\alpha'+1$, where $\alpha'$ is a successor ordinal. Since $M_{\alpha'}\prec H(\Theta)$ and $(M_i)_{i<\alpha'}\in M_{\alpha'}$, we can find a sequence $(x_i)_{i<\alpha'}\in M_{\alpha'}$ be the inductive assumption (again using Lemma \ref{Lemma: Skolem Hull Presentation} to make sure that $M_{\alpha'}$ can define the Skolem hulls) and elementarity. Then just find $x_{\alpha'}$ as in the case $\alpha=1$. It follows that $(x_i)_{i<\alpha}$ is as required.
			
			Lastly, let $\alpha=\alpha'+1$, where $\alpha'$ is a limit ordinal. Let $(\alpha_n)_{n\in\omega}$ be an increasing sequence of successor ordinals converging to $\alpha'$, where, for notational simplicity, $\alpha_0=0$. As in the case $\alpha=1$, find a countable set $x=\{\eta_0,\eta_1,\dots\}$ such that, for every $i<\alpha$, $\Hull^{\mathcal{A}}(M_i\cup x)\cap\omega_1=M_i\cap\omega_1$ and such that $\sup(\Hull^{\mathcal{A}}(M_{\alpha'}\cup x)\cap\omega_2)\in S$.  However, $(\Hull^{\mathcal{A}}(M_i\cup x))_{i\in \alpha}$ is not an $\alpha$-tower, some more work is needed.
			
			For every $n\in\omega$,
			$$(\Hull^{\mathcal{A}}(M_i\cup\{\eta_0,\dots,\eta_{n-1}\}))_{\alpha_n\leq i<\alpha_{n+1}}\in\Hull^{\mathcal{A}}(M_{\alpha_{n+1}}\cup\{\eta_0,\dots,\eta_{n-1}\})$$ by Lemma \ref{Lemma: Skolem Hull Presentation}. By applying the inductive assumption to this sequence inside of $\Hull^{\mathcal{A}}(M_{\alpha_{n+1}}\cup\{\eta_0,\dots,\eta_{n-1}\})$, we obtain a sequence $(y_i)_{\alpha_n\leq i<\alpha_{n+1}}$ such that, for every $i\in[\alpha_n,\alpha_{n+1})$,
			\begin{align*}
				\Hull^{\mathcal{A}}(M_i\cup \{\eta_0,\dots,\eta_{n-1}\}\cup y_i)\cap\omega_1 & =\Hull^{\mathcal{A}}(M_i\cup\{\eta_0,\dots,\eta_{n-1}\})\cap\omega_1 \\
				& = M_i\cap\omega_1
			\end{align*}
			and
			$$\sup\left(\Hull^{\mathcal{A}}(M_i\cup\{\eta_0,\dots,\eta_{n-1}\}\cup y_i)\cap\omega_2\right)\in S$$
			such that $(y_j)_{\alpha_n\leq j<i}\in \Hull^{\mathcal{A}}(M_i\cup\{\eta_0,\dots,\eta_{n-1}\}\cup y_i)$ for $i$ successor and $\Hull^{\mathcal{A}}(M_i\cup \{\eta_0,\dots,\eta_{n-1}\}\cup y_i)=\bigcup_{\alpha_n\leq j<i}\Hull^{\mathcal{A}}(M_j\cup \{\eta_0,\dots,\eta_{n-1}\}\cup y_j)$ for $i$ limit.
			
			Define, for $i\in[\alpha_n,\alpha_{n+1})$, $x_i:=y_i\cup\{\eta_0,\dots,\eta_{n-1}\}$ and $x_{\alpha'}:=\{\eta_0,\eta_1,\dots\}$. We show that $(x_i)_{i<\alpha}$ is as required.			
			
			Let $i<\alpha$ be a successor. We show that $(x_j)_{j<i}\in\Hull^{\mathcal{A}}(M_i\cup x_i)$: First suppose that $i=\alpha_n$ for some $n$. Then we have that for every $n'<n$, $(y_j)_{\alpha_{n'}\leq j<\alpha_{n'+1}}\in M_{\alpha_{n'+1}}\subseteq M_{\alpha_n}$. Therefore, the sequence $(y_j)_{j<\alpha_n}=(y_j)_{j<i}\in M_i$. However, from this sequence we can easily define $(x_j)_{j<i}$.
			
			Now suppose that $\alpha_n<i<\alpha_{n+1}$. By the previous argumentation, $(x_j)_{j<\alpha_n}\in M_{\alpha_n}\subseteq M_i$. Additionally, by assumption, $(y_j)_{\alpha_n\leq j<i}\in M_i$. As before, this allows us to construct $(x_j)_{\alpha_n\leq j<i}$ inside $M_i$ and thus $(x_j)_{j<i}$.
			
			Lastly, let $i<\alpha$ be a limit. If $i<\alpha'$, $\Hull^{\mathcal{A}}(M_i\cup x_i)=\bigcup_{j<i}\Hull^{\mathcal{A}}(M_j\cup x_j)$ holds by assumption. If $i=\alpha'$, we have
			$$\Hull^{\mathcal{A}}(M_i\cup x_i)=\bigcup_{j<i}\Hull^{\mathcal{A}}(M_j\cup x_j)$$
			since the right-hand union, which is contained in $M_{\alpha'}$, is an elementary substructure of $\mathcal{A}$ containing $M_i$ and $x_i$.
		\end{proof}
		
		Now the theorem follows easily, since in the above situation, $(\Hull^{\mathcal{A}}(M_i\cup x_i))_{i<\alpha}$ is the required $\alpha$-tower, again thanks to Lemma \ref{Lemma: Skolem Hull Presentation}.
	\end{proof}
	
	As a corollary to Theorem \ref{Theorem: Extend alpha-tower into stationary set}, we can show that in many natural situtations, $\dP(S)$ is $(\alpha,\omega_1)$-semiproper:
	
	\begin{mypro}
		Suppose that $\alpha<\omega_1$ is an ordinal such that player II has a winning strategy in $\Game_{\omega_1}^{\omega\alpha}(I_{\text{bdd}}^{\cof(\omega_1)})$. Let $S\subseteq E_{\omega}^{\omega_2}$ be stationary. Then $\dP(S)$ is $(\alpha,\omega_1)$-semiproper.
	\end{mypro}
	
	\begin{proof}
		Let $(M_i)_{i<\alpha}$ be an $\alpha$-tower of countable elementary substructures of $H(\Theta)$. By Theorem \ref{Theorem: Extend alpha-tower into stationary set}, we can find an $\alpha$-tower $(M_i')_{i<\alpha}$ pointwise end-extending $(M_i)_{i<\alpha}$ such that for every $i<\alpha$, $\sup(M_i\cap\omega_2)\in S$. Let $p\in\dP(S)\cap M_0$.
		
		\begin{myclaim}
			There exists a condition $q\in\dP(S)$ with $q\leq p$ such that $q$ is $(M_i',\dP(S))$-generic for every $i<\alpha$.
		\end{myclaim}
		
		\begin{proof}
			We define a descending sequence $(p_i)_{i<\alpha}$ of conditions in $\dP(S)$ such that, for every $i<\alpha$, the following holds:
			\begin{enumerate}
				\item $p_i$ is $(M_i',\dP(S))$-generic;
				\item $\dom(p_i)=(M_i'\cap\omega_1)+1$ and $p_i(M_i\cap\omega_1)=\sup(M_i'\cap\omega_2)$;
				\item $p_i\in M_{i+1}'$ and if $i$ is a successor, $(p_j)_{j<i}\in M_i'$.
			\end{enumerate}
			
			Let $i$ be a successor ordinal or $0$ (we let $p_{-1}:=p$ for notational convenience). Let $(D_n)_{n\in\omega}$ enumerate all open dense subsets of $\dP(S)$ lying in $M_i$. Define a descending sequence $p_i^n$ of conditions extending $p_{i-1}$ such that $p_i^n\in D_n\cap M_i'$. It follows easily that if we let $p_i':=\bigcup_{n\in\omega}p_i^n$, $\dom(p_i')=M_i'\cap\omega_1$ and $\im(p_i')$ is a cofinal subset of $\sup(M_i'\cap\omega_2)$. Thus, $p_i:=p_i'\cup\{(M_i'\cap\omega_1,\sup(M_i'\cap\omega_2))\}\in\dP(S)$ and is as required.
			
			Let $i\leq\alpha$ be a limit ordinal. Let $p_i':=\bigcup_{j<i}p_j$. It follows as before that $\dom(p_i')=M_i'\cap\omega_1$ and $\im(p_i')$ is a cofinal subset of $\sup(M_i'\cap\omega_2)$. Thus, $p_i:=p_i'\cup\{(M_i'\cap\omega_1,\sup(M_i'\cap\omega_2))\}$ is as required.
			
			Lastly, $q:=p_{\alpha}$ is the required condition.
		\end{proof}
		
		Now it follows that $q$ is $(M_i)_{i<\alpha}$-semigeneric: Let $G$ be any $\dP(S)$-generic filter containing $q$. Let $\tau\in M_i$ be a name for a countable ordinal. Then $\tau\in M_i'$ and, since $q$ is $(M_i',\dP(S))$-generic, $\tau_G\in M_i'$. But $M_i'\cap\omega_1=M_i\cap\omega_1$, so $\tau_G\in M_i$.
	\end{proof}
	
	\section{The Main Theorems}\label{section: maintheorem}
	
	In this section, we prove the main theorems. As a warm-up, which does not use revised countable support iterations, we first show that it is consistent that every $\omega$-bounded coloring on pairs in $\omega_2$ has a rainbow subset of order type $\omega_1$ stationary in its supremum.
	
	\begin{mysen}\label{Theorem: Stationary Rainbow Set}
		Suppose that $\kappa$ is a cardinal such that there are stationarily many weakly compact cardinals below $\kappa$. Let $G$ be $\Coll(\omega_1,<\kappa)$-generic.
		
		In $V[G]$, $\kappa\to^*(\omega_1\text{-st})_{\omega\text{-bdd}}^2$.
	\end{mysen}
	
	\begin{proof}
		For $\nu<\kappa$, denote by $G\uhr\nu$ the $\Coll(\omega_1,<\nu)$-generic filter induced by $G$.
		
		Let $S\subseteq\kappa$ be the set of weakly compact cardinals below $\kappa$ in $V$. Suppose that, in $V[G]$, $f\colon[\kappa]^2\to\kappa$ is an $\omega$-bounded coloring on $\kappa$. Let $M\prec H(\Theta)$ be such that $|M|<\kappa$, $M^{\omega_1}\subseteq M$, $f\in M$ and $\nu:=M\cap\kappa\in S$.
		
		Clearly, since $f$ is $\omega$-bounded, $f_{\nu}:=f\uhr[\nu]^2 \in V[G\restriction \nu]$ is $\omega$-bounded as well. Furthermore, since $\nu$ is weakly compact, in $V[G\uhr\nu]$, the hypotheses of Theorem \ref{Theorem: Adding Rainbow Set} are satisfied by Theorem \ref{Theorem: Ideal Game Upper Bound}, so there exists a poset $\dP_{f_{\nu}}$ which is countably closed, preserves $\nu$ and partitions $\nu$ into $\omega_1$ many $f_{\nu}$-rainbow sets.
		
		Since $\Coll(\omega_1,<\nu)*\dP_{f_{\nu}}$ is countably closed, the following holds (see e.g. \cite[Theorem 14.3]{CummingsHandbook}):
		
		\begin{myclaim}
			There is a complete embedding $\iota\colon\Coll(\omega_1,<\nu)*\dP_{f_{\nu}}\to\Coll(\omega_1,<\kappa)$ extending the identity embedding from $\Coll(\omega_1, <\nu)$ to $\Coll(\omega_1, <\kappa)$ such that the quotient is forcing equivalent to $\Coll(\omega_1,<\kappa)$.
		\end{myclaim}
		
		In particular, in $V[G]$, there is a generic $G\restriction{\nu}*H$ for $\Coll(\omega_1,<\nu)*\dP_{f_{\nu}}$ such that $V[G]$ is an extension of $V[G\restriction{\nu}*H]$ using $\Coll(\omega_1,<\kappa)^{V[G_{\nu}*H]}$. Clearly, in $V[G\restriction{\nu}*H]$, there is a partition of $\nu$ into $\omega_1$ many $f_{\nu}$-rainbow sets and thus, in particular, there is a stationary $f_{\nu}$-rainbow set $A\subseteq\nu\cap\cof(\omega)$. Since $V[G]$ is an extension of $V[G\restriction{\nu}*H]$ using a countably closed -- hence proper -- forcing, $A\subseteq\nu\cap\cof(\omega)$ is still stationary in $\nu\cap\cof(\omega)$ in $V[G]$. In particular, since $f_{\nu}=f\uhr[\nu]^2$, $A$ is an $f$-rainbow set which is stationary in $\nu=\sup(A)$.
	\end{proof}
Theorem \ref{theorem: main2} follows easily from Theorem \ref{Theorem: Stationary Rainbow Set}.
\begin{mybem}
Although Theorem \ref{Theorem: Stationary Rainbow Set} improves significantly the large cardinal upper of the main theorem from \cite{GartiZhang}, it does not recover the the most general form of the theorem: in \cite{GartiZhang}, $\omega_2\to^* (\omega_1-st)_{<\omega_2-t-bdd}^2$ (recall Definition \ref{deftbd}) is shown consistent relative to suitable large cardinals. By Proposition \ref{proposition: typebounded}, our method falls short of dealing with a coloring that is $<\omega_1^\omega$-type-bounded. It is not too hard to check that our proof does work for colorings that are $<\omega_1^n$-type bounded for $n\in \omega$.
\end{mybem}

\begin{mybem}
By \cite[Remark 3.2]{GartiZhang}, $\omega_2\to^* (\omega_1-st)_{\omega-bdd}^2$ implies $\omega_2^V$ is a Mahlo cardinal in $L$, by a result of Jensen.
\end{mybem}

	With considerably more work, we can improve the preceding theorem to instead obtain a \emph{closed} rainbow set. This requires interleaving the club-shooting poset $\dP(A)$ into the iteration. Therefore, the quotient forcing is no longer countably closed, but merely $(\alpha,\omega_1)$-semiproper for all $\alpha\in \omega_1$ which was the reason for all of the work we had to perform to ensure that, in this situation, we can still add $f$-rainbow sets.
	
	To ensure that the continuum hypothesis holds in the extension, we use a concept due to Jensen (\cite{JensenSubcompleteL}) known as \emph{subcompletness}. The definition of subcompleteness is given in \cite[Page 31]{JensenSubcompleteL}. In \cite[Theorem 3]{JensenSubcompleteL}, it is shown that being subcomplete is preserved under $\RCS$-iterations. Lastly, in \cite[Lemma 6.3]{JensenSubcompleteL}, Jensen shows that the poset $\dP(S)$ is subcomplete. As a byproduct, the model we obtain is ``maximal" in a sense.
	
	\begin{mysen}\label{theorem: finalforcing}
		Suppose that $\kappa$ is a supercompact cardinal. There is a forcing extension where 
	\begin{itemize}
	\item $\omega_2\to^*(\omega_1-cl)_{\omega-bdd}^2$ holds and
	\item $\mathrm{MA}_{\omega_1}(\Gamma)$ holds where $\Gamma$ is the collection of forcings that are subcomplete and $(\alpha,\omega_1)$-semiproper for every $\alpha\in \omega_1$.
\end{itemize}			
		
	\end{mysen}
	\begin{proof}
		Let $l\colon\kappa\to V_{\kappa}$ be a Laver function. By Theorem \ref{Theorem: Copy of Defining RCS Iteration}, let $\overline{\dB}:=(\dB_i)_{i<\kappa+1}$ be an $\RCS$ iteration such that the following holds:
		\begin{enumerate}
			\item If $i<\kappa$ is inaccessible, $|\dB_j|<i$ for every $j<i$ and $l(i)$ is a $\dB_i$-name for a poset which is subcomplete and $(\alpha,\omega_1)$-semiproper  for every $\alpha<\omega_1$, let $\dB_{i+1}:=\dB_i*l(i)$;
			\item otherwise, let $\dB_{i+1}:=\dB_i*\Coll(\omega_1,|2^{\dB_i}|)$.
		\end{enumerate}
		
		By Theorem \ref{Theorem: Copy of RCS Iteration}, $\dB_{\kappa}$, which is the direct limit of $(\dB_i)_{i<\kappa}$, is $(\alpha,\omega_1)$-semiproper for every $\alpha<\omega_1$ and in particular preserves $\omega_1$. Thus, $\dB_{\kappa}$ forces $\kappa=\omega_2$.
		
		\begin{myclaim}
			In $V[G]$, there exists a $<\kappa$-complete ideal $I$ on $\kappa$ such that player II wins $\Game^{\alpha}(I)$ for every $\alpha<\omega_1$.
		\end{myclaim}
		
		\begin{proof}
			Let $j\colon V\to M$ be a measurable embedding with critical point $\kappa$ such that $j(l)(\kappa)=\Coll(\omega_1,\kappa)$. Hence, $j(\overline{\dB})_{\kappa+1}=\dB_{\kappa}*\Coll(\omega_1,\kappa)$ and it forces that \linebreak $j(\overline{\dB})_{j(\kappa)+1}/j(\overline{\dB})_{\kappa+1}$ is $(\alpha,\omega_1)$-semiproper for every $\alpha<\omega_1$. By Lemma \ref{Lemma: Two-Step Iteration}, it follows that $\dB_{\kappa}$ forces  that $j(\dB_{\kappa})/\dB_{\kappa}$ is $(\alpha,\kappa)$-semiproper for any $\alpha\in \omega_1$. Now the claim follows from Lemma \ref{Lemma: Ideal Existence}.
		\end{proof}
		
		In $V[G]$, let $f$ be an $\omega$-bounded coloring on $[\kappa]^2$. It follows from work of Jensen (see \cite{JensenSubcompleteL}) that $\dB_{\kappa}$ does not add reals. Thus, by Theorem \ref{Theorem: Adding Rainbow Set}, in $V[G]$, there is a countably closed poset $\dP_f$ which preserves $\kappa$ and adds a partition of $\kappa$ into $\omega_1$ many $f$-rainbow sets. In particular, $\dP_f$ adds a stationary $f$-rainbow set $A\subseteq\kappa\cap\cof(\omega)$. Let $\dot{A}$ be a $\dP_f$-name for $A$ and let $\dQ_f:=\dP_f*\dP(\dot{A})$. Then $\dQ_f$ preserves stationary subsets of $P_{\omega_1}^{\alpha}(\omega_1)$ for every $\alpha\in \omega_1$ by Theorem \ref{Theorem: Alpha-Stationary Preservation}. By Lemma \ref{Lemma: Dagger for alpha-stationarity}, $\dQ_f$ is $(\alpha,\omega_1)$-semiproper for every $\alpha<\omega_1$.
		
		Now let $j\colon V\to M$ be a $|\dQ_f|$-supercompact embedding such that $j(l)(\kappa)=\dQ_f$. It follows that $j(\overline{\dB})_{\kappa+1}=\dB_{\kappa}*\dQ_f$. Furthermore, $j$ lifts to $j^+\colon V[G]\to M[H]$, where $H$ is $j(\overline{\dB})$-generic. Thus, in $M[H]$, there is a closed $f$-rainbow set $c\subseteq\kappa$ with ordertype $\omega_1$. 
Note that $j^+(f)\restriction [\kappa]^2 =f$.		 Thus, in $M[H]$, there is a closed $j(f)$-rainbow set of order type $\omega_1$. By the  elementarity of $j^+$, in $V[G]$, there is a closed $f$-rainbow set $d\subseteq\kappa$ with ordertype $\omega_1$. This finishes the proof.
	\end{proof}
	
	Theorem \ref{theorem: main} follows easily from Theorem \ref{theorem: finalforcing}.
	\section{Open Questions} \label{section: questions}
	
	We close with a few open questions.

	\begin{myque}
		What is the consistency strength of the assertion that, whenever $c\colon[\omega_2]^2\to\omega_2$ is $\omega$-bounded, there is a $c$-rainbow set of ordertype $\omega_1$ stationary (or closed) in its supremum?
	\end{myque}

	In our construction, the main issue raising the consistency strength to some degree of supercompactness is the reliance on Lemma \ref{Lemma: Dagger for alpha-stationarity}. This could be vastly improved by a positive answer to the following:
	
	\begin{myque}
	For any $\omega$-bounded $f: [\omega_2]^2\to \omega_2$, does $<\omega_1$-$\mathrm{SCC}^{\cof}$ imply that $\dQ_f$ is $(\alpha,\omega_1)$-semiproper for all $\alpha<\omega_1$? Here $\dQ_f$ is the forcing defined in the proof of Theorem \ref{theorem: finalforcing}. More specifically, $\dQ_f = \dP_f * \dP(\dot{A})$ where $\dot{A}$ is a $\dP_f$-name for a $f$-rainbow stationary subset of $E^{\omega_2}_\omega$.
	\end{myque}	
	
	It is not hard to see that the methods developed in Section \ref{Section: Adding Rainbow Sets} can be straightforwardly adapted to larger successors of regular cardinals. In particular, as in Theorem \ref{Theorem: Stationary Rainbow Set}, for every regular cardinal $\mu$, it is consistent that for every $\mu$-bounded coloring $f\colon[\mu^{++}]^2\to\mu^{++}$, there is a stationary $f$-rainbow set of ordertype $\mu^+$. However, due to the reliance on $(\alpha,\omega_1)$-semiproperness, the same cannot be said for \emph{closed} rainbow sets. Thus, we ask:
	
	\begin{myque}
		Is it consistent that whenever $f\colon[\omega_3]^2\to\omega_3$ is $2$-bounded, there is a closed $f$-rainbow set of ordertype $\omega_2$?
	\end{myque}
	
	As stated in Section \ref{section: limitations}, our techniques are also unable to deal with type-bounded colorings, since we are provably unable to add rainbow sets without collapsing $\omega_2$. This begs the following question:
	
	\begin{myque}
		Is it consistent that for any $<\omega_2$-type bounded coloring $c: [\omega_2]^2\to \omega_2$, there exists a \emph{closed} $c$-rainbow $A\subseteq \omega_2$ of order type $\omega_1$?
	\end{myque}	
	
	Lastly, we need to iterate forcings that are $(\alpha,\omega_1)$-semiproper for all $\alpha\in \omega_1$ in order to obtain a winning strategy for player II in $\Game^{\alpha}(I_{\text{bdd}}^{\cof(\omega_1)})$ in our final model. In particular, the following is unclear:
	
	\begin{myque}
		Does Martin's Maximum decide the truth of $\omega_2\to^*(\omega_1-cl)_{\omega\text{-bdd}}^2$?
	\end{myque}

	\section{Acknowledgment}
	The second author wants to thank Steve Jackson for illuminating conversations on the family of games $\Game^\alpha(I)$.
	
	\printbibliography
	
\end{document}